\documentclass[leqno,10pt]{article}
\usepackage{amsmath}
\usepackage{amsfonts}
\usepackage{array}
\usepackage{enumerate}
\usepackage{graphicx}
\usepackage{amssymb}
\usepackage{amsthm}
\usepackage{xcolor}
\usepackage{chemarr}

\numberwithin{equation}{section}

\newtheorem{theorem}{Theorem}
\newtheorem{proposition}[theorem]{Proposition}

\newtheorem{definition}[theorem]{Definition}

\newtheorem{remark}{Remark}
\newtheorem{example}{Example}
\newcommand{\R}{\mathbb{R}}

\DeclareMathOperator{\diag}{diag}
\def\tr{\mathrm{tr}}

\newcommand{\norm}[2][\relax]{\ifx#1\relax \ensuremath{\left\Vert#2\right\Vert} \else \ensuremath{\left\Vert#2\right\Vert_{#1}}\fi}

\begin{document}

\title{Singular perturbations and scaling}

\author{Christian Lax, Sebastian Walcher\\
Mathematik A, RWTH Aachen \\
52056 Aachen, Germany\\
}


\maketitle
\begin{abstract} { Scaling transformations involving a small parameter ({\em degenerate scalings}) are frequently used for ordinary differential equations that model (bio-) chemical reaction networks. They are motivated by quasi-steady state (QSS) of certain chemical species, and ideally lead to slow-fast systems for singular perturbation reductions, in the sense of Tikhonov and Fenichel. In the present paper we discuss properties of such scaling transformations, with regard to their applicability as well as to their determination. Transformations of this type are admissible only when certain consistency conditions are satisfied, and they lead to singular perturbation scenarios only if additional conditions hold, including a further consistency condition on initial values. Given these consistency conditions, two scenarios occur. The first (which we call standard) is well known and corresponds to a classical quasi-steady state (QSS) reduction. Here, scaling may actually be omitted because there exists a singular perturbation reduction for the unscaled system, with a coordinate subspace as critical manifold. For the second (nonstandard) scenario scaling is crucial. Here one may obtain a singular perturbation reduction with the slow manifold having dimension greater than expected from the scaling. For parameter dependent systems we consider the problem to find all possible scalings, and we show that requiring the consistency conditions allows their determination. This lays the groundwork for algorithmic approaches, to be taken up in future work. In the final section we consider some applications. In particular we discuss relevant nonstandard reductions of certain reaction-transport systems.}

{\bf MSC (2010):} 92C45, 34E15, 80A30, 13P10  \\
{\bf Key words}: Reaction networks, dimension reduction, invariant sets, critical manifold.\\

\end{abstract}

\section{Introduction}

{In the mathematical analysis of reaction networks, degenerate scaling (i.e. scaling with a small parameter) is frequently applied to set quasi-steady state (QSS) behavior of certain species in a proper mathematical framework: Here, Tikhonov's classical theorem \cite{tikh} on singular perturbations may become applicable for the scaled system. It seems that Heineken, Tsuchiya and Aris \cite{hta} were the first to rigorously apply this strategy to a biochemical reaction network (the Michaelis-Menten system), and this laid the groundwork for many further applications. Generally, singular perturbation methods (in a geometric framework introduced by Fenichel \cite{feninv,fenichel}) provide a powerful tool to reduce the dimension of parameter-dependent differential equation systems. \\
Scaling methods comprise just one approach to search for singular perturbation scenarios. In chemical reaction networks they correspond to QSS for certain chemical species, but partial equilibrium approximations (PEA, motivated by slow and fast reactions) are equally relevant; see e.g. Heinrich and Schauer \cite{hs}, Goussis \cite{gouQSS}. From a purely mathematical perspective, a method to determine all critical parameter values for singular perturbation reduction of polynomial or rational ODE systems, as well as to compute the reduced systems, was recently introduced in \cite{gw2} and \cite{gwz}. This method requires no a priori input, such as stipulating or guessing slow and fast variables or reactions, but (as should be expected) there are feasibility problems when dimensions increase. Thus, searching for singular perturbation reductions under sensible a priori restrictions (e.g. inspired by chemistry) remains relevant, in particular for higher dimensions. Moreover, recent work by Noel et al. \cite{noelgvr}, Samal et al. \cite{sgfr,sgfwr} has revived interest in scaling transformations. The work in \cite{noelgvr,sgfr,sgfwr}  is based on an interpretation of QSS as cancellation of fast reaction terms, and allows to employ methods from tropical geometry to identify slow and fast variables. While details are intricate, considering only terms of lowest order in the small parameter yields a degenerate scaling in the above mentioned sense. \\}

{The present article contains a detailed analysis of degenerate scalings and their applications. We recall some facts and introduce some notation in Section 2. In Section 3 we present a detailed (and elementary) discussion of necessary conditions for admissibility of degenerate scalings. These conditions include  local consistency requirements (regarding the existence of certain invariant sets, resp. the existence of certain manifolds of equilibria) but also a requirement on initial values for the scaled variables. In the remainder of the paper these consistency conditions will be assumed and utilized.
We proceed to discuss reductions, with two scenarios of interest. For the first (``standard'') case, Heineken et al. \cite{hta} is a paradigmatic example. But in this standard case scalings are not a necessary prerequisuite for singular perturbation reductions; indeed Fenichel's theory \cite{fenichel} is directly applicable. However, the situation is different in the second (``nonstandard'') scenario, which seems to have received less attention: Scaling is necessary here, and for the scaled system one obtains a singular perturbation reduction, with a slow manifold of dimension greater than one may expect from the scaling.
In Section 4 we use the consistency conditions to determine all possible degenerate scalings of a given parameter dependent system, thus laying a foundation for algorithmic approaches in future work.
In Section 5 we discuss applications to PTM networks and to reaction transport systems. The latter illustrate that nonstandard scenarios appear in relevant applications. Some facts from \cite{gw2} are collected in Appendix A, for the reader's convenience. A few new results are also included, with proofs. Appendix B contains some computations.}

\section{Motivation and basic notions}
Our discussion starts from a system 
\begin{equation}\label{specpar}
\dot z = \widehat h(z,\varepsilon)=\widehat h^{(0)}(z)+\varepsilon \widehat h^{(1)}(z)+\cdots
\end{equation}
with $z\in \mathbb R^n$ and  a distinguished real ``small parameter'' $\varepsilon$. The system may depend on further parameters which are assumed to be be fixed and will be suppressed in the notation until further notice. For the sake of simplicity we always assume that $\widehat h$ is defined and smooth on an open subset of $\mathbb R^n\times\mathbb R$ that contains some point $(z_0,0)$. We will later specialize to polynomial and rational systems, due to our primary interest in reaction networks.\\
Consider a partitioning 
\begin{equation}\label{parteq}
z=\begin{pmatrix}x\\ y\end{pmatrix} \text{  with  }x\in \mathbb R^r,\, y\in\mathbb R^s,\,r\geq1,\,s\geq1,\,r+s=n, 
\end{equation}
and accordingly rewrite \eqref{specpar} with initial value $z_0=(x_0,y_0)^{\rm tr}$ as a system
\begin{equation}\label{sepsys}
\begin{array}{rccccl}
\dot x&=& f(x,y,\varepsilon)&=&f_0(x,y)+\varepsilon f_1(x,y)+\varepsilon^2\cdots, &x(0)=x_0\\
\dot y&=& g(x,y,\varepsilon)&=&g_0(x,y)+\varepsilon g_1(x,y)+\varepsilon^2\cdots,& y(0)=y_0,
\end{array}
\end{equation}
defined in a neighborhood of $U\times V\times\{0\}$,
where $U\subseteq\mathbb R^r$, $V\subseteq\mathbb R^s$ have nonempty interior, and $0\in {\rm int}\,V$. We are interested in the asymptotic behavior of the system as $\varepsilon\to 0$ in {\em singular settings} (according to Fenichel's \cite{fenichel} working definition), thus $\widehat h^{0}$ has non-isolated zeros.\\
We recall some pertinent facts from singular perturbation theory; see e.g. Verhulst \cite{verhulst}. Assume that system \eqref{sepsys} is in \textit{Tikhonov standard form}, and the partitioning separates slow and fast variables, thus
\begin{equation}\label{tikhostandard}
\begin{array}{rcccl}
\dot x&=& &  &\varepsilon\widetilde f_1(x,y)+\varepsilon^2\cdots,\quad x(0)=x_0\\
\dot y&=&\widetilde  g_0(x,y)&+&\varepsilon\widetilde g_1(x,y)+\varepsilon^2\cdots,\quad y(0)=y_0,\\
\end{array}
\end{equation}  
where $\varepsilon >0$. 
One may rewrite the system in {\em slow time} $\tau =\varepsilon t$ in the form
\begin{equation}\label{tikhostandardslowtime}
\begin{array}{rcccl}
x^\prime &=& &  &\widetilde f_1(x,y)+\varepsilon\cdots,\quad x(0)=x_0\\
\varepsilon y^\prime&=&\widetilde g_0(x,y)&+&\varepsilon\widetilde g_1(x,y)+\varepsilon^2\cdots,\quad y(0)=y_0.\\
\end{array}
\end{equation}
We furthermore assume that at some point of $U\times V$ the hypotheses of the implicit function theorem hold, hence the equation $\widetilde g_0(x,y)=0$  locally admits a unique solution in the form  $y=\phi(x)$, whence the zero set of $\widetilde g_0$ contains an $r$-dimensional submanifold $Z$, called the critical manifold (or asymptotic slow manifold). Every point of $Z$ is stationary for the system at $\varepsilon=0$. Classical results by Tikhonov \cite{tikh} and Fenichel  \cite{fenichel} show the existence of an asymptotic reduction as $\varepsilon\to 0$, given suitable conditions on $U$, $V$ and $\widetilde g_0$; see Verhulst \cite{verhulst}, Thm.~8.1ff.   {We impose somewhat stronger conditions than in \cite{verhulst}, requiring that for some $\nu>0$ all eigenvalues of the partial derivative $D_2\widetilde g_0(x,y)$, with $(x,y)\in Z$, have real parts $\leq -\nu$; thus $Z$ is locally exponentilly attracting.  }
Then, by results of O'Malley/Vasil'eva (see \cite{verhulst}, Thm.~8.2), there exist a neighborhood of $Z$ and some $T_2>0$ such that for any initial value $(x_0,y_0)$ in this neighborhood, the solution of \eqref{tikhostandard} converges uniformly to the solution of the reduced initial value problem
\[
x^\prime =\widetilde  f_1(x,\phi(x)),\quad x(0)=x_0
\]
in slow time $\tau$ on any interval $\left[T_1, \,T_2\right]$ with $0< T_1<T_2$, as $\varepsilon\to 0$. \\

In many applications a parameter dependent system is not initially written in Tikhonov standard form. Then one may attempt to scale certain variables by a parameter $\varepsilon>0$ to obtain a system in standard form, and then let $\varepsilon\to 0$. As an illustration (and as a benchmark later on) we consider the well-known irreversible Michaelis-Menten system
\begin{equation}\label{mmirr}
\begin{array}{rcl}
\dot{s}&=&-k_1es+k_{-1}c\\
\dot{e}&=&-k_1es +(k_{-1}+k_2)c\\
\dot{c}&= &k_1es -(k_{-1}+k_2)c\\
\end{array}
\end{equation}
with (typical) initial values $s(0)=s_0>0$, $c(0)=0$; $e(0)=e_0>0$. From the first integral $e+c$ one obtains the two-dimensional system
\begin{equation}\label{mmirr2d}
\begin{array}{rcl}
\dot{s}&=&-k_1e_0s+(k_1s+k_{-1})c\\
\dot{c}&= &k_1e_0s -(k_1s+k_{-1}+k_2)c.\\
\end{array}
\end{equation}
Credit for a mathematically rigorous reduction by scaling of this system is due to Heineken, Tsuchiya and Aris \cite{hta}. The technique was subsequently used in numerous further applications; see e.g.~Murray's monograph \cite{Mur} on mathematical biology for many examples. We outline the essential features of the reduction to motivate the following sections.
\begin{example}\label{mmex}
{\em 
Consider \eqref{mmirr2d} for small initial enzyme concentration $e_0=\varepsilon e_0^*$.
\begin{itemize}
\item The system  is not in Tikhonov standard form, but upon replacing $c$ by $c^*:= c/\varepsilon$ one obtains the standard form
\[
\begin{array}{rcl}
\dot{s}&=&\varepsilon (-k_1se_0^*+(k_1s+k_{-1})c^*)\\
\dot{c^*}&= &k_1se_0^* -(k_1s+k_{-1}+k_2)c^*.\\
\end{array}
\]
Tikhonov's theorem yields the familiar reduction to the one-dimensional equation $\dot s=-k_1k_2e_0s/(k_1s+k_{-1}+k_2)$.
\item The scaling $c=\varepsilon c^*$ becomes non-invertible in the limit $\varepsilon\to 0$. This degeneracy does not affect the validity of the reduction (since passing back from $c^*$ to $c$ is unproblematic), but the crucial point is that by passing from $c$ to $c^*$ one still has a well-defined smooth system.
\item Scaling is not necessarily related to singular perturbations. Its primary purpose is to nondimensionalize the system and to eliminate unnecessary parameters. According to this guideline, Heineken et al.  \cite{hta} introduce further scalings, e.g. $s= s^*\cdot s_0$, and $\varepsilon=e_0/s_0$. In the course of these scalings the ``small parameter'' appears in a natural way. 
(For more about scaling see also Murray \cite{Mur}, Ch.~6, Sections 6.1--6.3, Segel and Slemrod \cite{SSl}.)
\item We reproduced the scaling transformations in \cite{hta} only in part, since we wanted to emphasize the role of the degeneracy in the scaling. On the other hand, scaling may obscure implicit occurrences of the small parameter, and thus may cause misconceptions in limiting processes. Indeed, the scaling $\mu=e_0/s_0$ from \cite{hta} yields $\mu\to 0$ whenever $s_0\to\infty$. However (with the notation of \cite{hta}, equation (10)) letting $s_0\to\infty$ and keeping the other parameters constant and positive implies that $\kappa\to 0$ and $\lambda\to 0$. Thus Tikhonov will yield a reduced equation with trivial right-hand side, instead of the familiar singular perturbation reduction. (This observation is not in disagreement with the direct convergence proof given in Segel and Slemrod \cite{SSl}, Section 6. Indeed, one of their assumptions on the parameters implies boundedness of $s_0$.)
\end{itemize}
}
\end{example}
\begin{definition}
Given a parameter $\varepsilon>0$, we call the linear transformation
\begin{equation}\label{partscale}
\begin{pmatrix}x\\ y\end{pmatrix}=\begin{pmatrix}x\\ \varepsilon y^*\end{pmatrix}
\end{equation}
an {\em asymptotically degenerate scaling} as $\varepsilon\to 0$ (briefly, a {\em degenerate scaling}) with respect to the partitioning \eqref{parteq}. 
\end{definition}
Degenerate scalings are abundant in the literature, in particular for reaction networks. Recently they appeared in connection with an algebraic approach to a systematic reduction of reaction networks:
\begin{example}\label{tropex}
\em{The tropical equilibration approach to identify quasi-steady state species in reaction networks was proposed and developed in Noel et al. \cite{noelgvr}, Samal et al. \cite{sgfwr} (see also Radulescu et al. \cite{rvg},  Samal et al. \cite{sgfr}). Given system \eqref{specpar}, the authors consider the dependence on $\varepsilon$ in more detail, letting $\widehat h$ be a finite linear combination of polynomial terms with integer powers of $\varepsilon$ as coefficients. (Rational powers are also permitted in \cite{sgfwr}, but this would amount to renaming the small parameter.) The choice of a scaling
\[
z_j=\widehat z_j\,\varepsilon^{a_j}
\]
is motivated by a particular mathematical interpretation of quasi-steady state: For variables in QSS the rate of change (i.e., the corresponding entry on the right-hand side of the scaled differential equation) should not be of lowest order in $\varepsilon$; hence a cancellation of dominant terms must occur. Consequently more than one dominant term must be present and this, in turn, provides conditions for the $a_j$ that are amenable to Newton polytope arguments. For the underlying assumptions and more details see Noel et al.~\cite{noelgvr}, Section 3, and in particular Lemma 3.3. (No explicit reference to Tikhonov is made in that paper; the connection is noted in Samal et al. \cite{sgfwr}.)\\
Carrying out the procedure sketched above yields a system 
\[
\varepsilon^{a_j}\dot{\widehat z_j}= H_j(\widehat z, \varepsilon),\quad 1\leq j\leq n,
\]
with rational numbers $a_j$. Redefining $\varepsilon$ and scaling time appropriately (via $\widehat t=\varepsilon^b\,t$), one may again assume that all $a_j$ are nonnegative integers, and at least one of them is equal to $0$. Setting
\[
\widetilde z_j:= \varepsilon^{a_j-1}\widehat z_j\text{  whenever  } a_j>0, \quad \widetilde z_j:= \widehat z_j \text{  whenever  } a_j=0,
\]
one arrives at a scaled system of the type
\[
\begin{array}{rcl}
\dot x&=&f_0(x,y)+o(1)\\
\varepsilon \dot y&=& g_0(x,y)+o(1).
\end{array}
\]
}
\end{example}

\section{Degenerate scalings}
In this section we discuss necessary and sufficient conditions such that \eqref{sepsys} can be transformed into Tikhonov standard form by a degenerate scaling $y=\varepsilon y^*$. Furthermore we provide conditions to ensure that the transformed system admits a singular perturbation reduction in the sense of Tikhonov and Fenichel, and we discuss the reductions.
\subsection{Consistency}
Replacing $y$ by $\varepsilon y^*$ in \eqref{sepsys} and rearranging yields
\begin{equation}\label{sepsysscal}
\begin{array}{rcccl}
\dot x&=& f_0(x,0)&+&\varepsilon (f_1(x,0)+D_2f_0(x,0)y^*)+\varepsilon^2\cdots\\
\dot y^*&=& \varepsilon^{-1}g_0(x,0)&+&(g_1(x,0)+D_2g_0(x,0)y^*)+\varepsilon\cdots \\
\end{array}
\end{equation}
with initial values $x(0)=x_0$, $y^*(0)=y_0^*:=\varepsilon^{-1}y_0$. The second equation, rewritten as
\[
\varepsilon \dot y^*=g_0(x,0)+\varepsilon (g_1(x,0)+D_2g_0(x,0)y^*)+\varepsilon^2\cdots,
\]
seems to fit the mold of \eqref{tikhostandardslowtime}. But the conditions for Tikhonov's theorem are never satisfied when $g_0(\cdot,\,0)\not=0$, since a resolution of the implicit equation in the form $y=\varphi(x)$ is impossible. Moreover there is a deeper reason for considering only the case $g_0(\cdot,\,0)=0$, as an elementary argument shows.
\begin{remark}\label{g0rem} {\em Assume that $g_0(x,0)$ is not identically zero. Then the scaling $y=\varepsilon y^*$ in \eqref{sepsys} produces solution components $y^*$ of \eqref{sepsysscal} that grow with order $\varepsilon^{-1}$ for arbitrarily small positive times, regardless of $\Vert y^*(0)\Vert$. This fact is readily verified by the continuous depencence theorem for the original system \eqref{sepsys}.\\
An different version of this argument is as follows: \eqref{sepsysscal} may be seen as the system
\[
\begin{array}{rcccl}
\frac{dx}{ds}&=&\varepsilon f_0(x,0)&+&\varepsilon^2\cdots\\
\frac{dy}{ds}&=& g_0(x,0)&+&\varepsilon\cdots \\
\end{array}
\]
rewritten in slow time $t=\varepsilon s$. Nonconstant solutions of the fast system (at $\varepsilon =0$) leave any bounded subset of $\mathbb R^s$.
}
\end{remark}
This simple but crucial remark motivates the first part of the following definition. For the second part compare system \eqref{tikhostandard} in standard form.
\begin{definition}\label{consdef}
\begin{enumerate}[(a)]
\item We call the scaling transformation $y=\varepsilon y^*$ (and system \eqref{sepsysscal}) {\em{locally consistent}} if $g_0(x,0)=0$ for all $x\in U$. 
\item We call the scaling transformation (and system \eqref{sepsysscal}) {\em{locally Tikhonov consistent}} if it is locally consistent and in addition $f_0(x,0)=0$ for all $x\in U$.
\end{enumerate}
\end{definition}
\begin{remark}\label{consrem}
\begin{enumerate}[(a)]
\item By a familiar lemma by Hadamard (see e.g. Nestruev \cite{Nes}), the identity $g_0(x,0)=0$ implies locally that $g_0(x,y)=G_0(x,y)y$, with smooth $G_0$ having values in $\mathbb R^{(s,s)}$. By the same token, we have $f_0(x,y)=F_0(x,y)y$ with $F_0$ having values in $\mathbb R^{(r,r)}$ whenever the identity $f_0(x,0)=0$ holds for all $x$.
\item Local consistency implies (and is by \cite{gwz3}, Lemma 2 equivalent to) invariance of the subspace defined by $y=0$ for system \eqref{sepsys} at $\varepsilon=0$.
\item One could weaken the local consistency condition by requiring only invariance of the zero set $W$ of $x\mapsto g_0(x,0)$ for the equation $\dot x=f_0(x,0)$. But such a choice would necessarily restrict any subsequent analysis to $W\times V$ and impose additional conditions. We will not pursue this further.
\end{enumerate}
\end{remark}
Local consistency (even local Tikhonov consistency) does not gurantee the existence of a singular perturbation reduction for the transformed system \eqref{sepsysscal}, even if eigenvalue conditions for the fast system hold. A second condition is needed.
\begin{definition}
 We call the scaling $y=\varepsilon y^*$ {\em{initial value consistent}} if the initial value for $y$ satisfies $y_0=\varepsilon y_0^*$, with fixed $y_0^*$.
\end{definition}
This condition ensures that the initial value $(x_0,y_0^*)$ remains in a fixed domain after scaling, which is a hypothesis of Tikhonov's theorem as stated in Verhulst \cite{verhulst}, Thm. 8.1. {We emphasize that this hypothesis is indeed necessary, and the formally reduced system may provide an incorrect limit as $\varepsilon\to 0$ if initial value consistency is absent.} The source of the problem lies in the fast dynamics.

\begin{example}\label{linex} 
{\em 
\begin{itemize}
\item Consider the  two-dimensional linear system
\begin{equation}\label{linex1}
\begin{array}{rcl}
\dot x &=&\varepsilon  ax + by\\
\dot y &=& cy
\end{array}
\end{equation}
with small parameter $\varepsilon$, constants $a\leq0$, $b$ arbitrary, and $c<0$, and initial values $x_0$ resp. $y_0$ at $t=0$. Scaling $y=\varepsilon y^*$ and passing to slow time $\tau$, one obtains 
\begin{equation}\label{linex2}
\begin{array}{rcl}
x' &=& ax + by^*\\
(y^*)' &=& \varepsilon^{-1} cy^*
\end{array}
\end{equation}
with initial values $x(0)=x_0$ and $y^*(0)=\varepsilon^{-1}y_0$. The scaled differential equation seems to be amenable to Tikhonov's theorem, but the ``escaping" initial value $y^*(0)$ leads to a noticeable discrepancy between the exact solution and the corresponding solution of the formally reduced equation: The first component of the solution of \eqref{linex2} equals
\[
x(\tau)=\frac{by_0}{c-\varepsilon a}\left(e^{\varepsilon^{-1}c\tau}-e^{a\tau}\right)+ x_0 e^{a\tau}
\]
while the formally reduced equation (given by $y^*=0$ and $x'=ax$) has the solution
\[
x_{\rm red}(\tau)= x_0 e^{a\tau}.
\]
Therefore
\[
x(\tau)-x_{\rm red}(\tau)=\frac{by_0}{c-\varepsilon a}\left(e^{\varepsilon^{-1}c\tau}-e^{a\tau}\right)\to -\frac{by_0}{c} e^{a\tau}\mbox{  as  }\varepsilon\to0,
\]
and $x(\tau)$ does not converge to $x_{\rm red}(\tau)$  in any interval $\left[T_1, \,T_2\right]$, $0< T_1<T_2$. The conclusion of Tikhonov's theorem does not hold.
\item This observation generalizes (with more involved computations) to linear systems
\begin{equation}\label{linex3}
\begin{array}{rc cl}
\dot x &=&  \varepsilon A x + By,&\quad x\in\mathbb R^r\\
\dot y &=&  \varepsilon Dx + Cy, &\quad y\in\mathbb R^s
\end{array}
\end{equation}
with matrices $A$ and $C$ of appropriate sizes, with all eigenvalues of $A$ having real parts $\leq 0$, all eigenvalues of $C$ having negative real parts, and matrices $B$, $D$ of appropriate sizes. The degenerate scaling $y^*=\varepsilon y$ yields correct reductions if (and generically only if) $y_0=\varepsilon y_0^*$ is of order $\varepsilon$. The argument can be extended further to nonlinear systems via Taylor expansions with respect to $y$.
\end{itemize}
}
\end{example}

For the remainder of the present paper we will always assume local consistency, as well as initial value consistency. Frequently we will furthermore require local Tikhonov consistency.\\

Due to the consistency requirements the list of possible scalings is subject to restrictions: Given the partitioning \eqref{parteq}, the zero set of the scaled variables must define an invariant subspace (or consist of stationary points only) at $\varepsilon =0$, and the initial values of these variables must be small.  We will discuss this matter in Section 4 below. With regard to the tropical equilibration ansatz (Example \ref{tropex}), local consistency and initial value consistency also impose additional conditions: It is only appropriate to scale those $\widehat z_j$ with $a_j>0$ and with small initial values. This observation may assist in simplifying the search.\\

\subsection{Reduction: The standard case}\label{secstandard}
The following Proposition includes the reduction from Example \ref{mmex} (and many others) as a special case. Part (a) is straightforward from Hadamard's lemma, and the proof of part (b) is a direct application of Tikhonov's theorem. A general proof of part (c) (which we include because one has to distinguish between singular perturbation and QSS reductions) is given in \cite{gwz3}, Proposition 5.
\begin{proposition}\label{standardapp}
\begin{enumerate}[(a)]
\item Whenever the scaling $y=\varepsilon y^*$ satisfies local Tikhonov consistency and initial value consistency then system \eqref{sepsysscal} locally has the form
\begin{equation}\label{sepsysscal2}
\begin{array}{rcccl}
\dot x&=& &+&\varepsilon (f_1(x,0)+F_0(x,0)y^*)+\cdots\\
\dot y^*&=& G_0(x,0)y^*+g_1(x,0) &+&\varepsilon\cdots \\
\end{array}
\end{equation}
with suitable matrix-valued functions $F_0$ and $G_0$ and initial conditions $x(0)=x_0$, $y^*(0)=y_0^*$.
\item Assume in addition that there exists $\nu>0$ such that all the eigenvalues of $G_0(x,0)$, $x\in U$, have real part $\leq -\nu$. 
Then the system admits a Tikhonov-Fenichel reduction to the asymptotic slow manifold
\[
Z=\{ (x,\,y^*);\,y^*=-G_0(x,0)^{-1}g_1(x,0)\},
\]
and there exist a neighborhood of $Z$ and $T_2>0$ such that any solution of \eqref{sepsysscal2} that starts in this neighborhood converges (in slow time $\tau=\varepsilon t$) to the solution of the reduced system 
\begin{equation}\label{redtnf}
x^\prime =f_1(x,0)-F_0(x,0)G_0(x,0)^{-1}g_1(x,0), \quad x(0)=x_0,
\end{equation}
uniformly on any interval $\left[T_1, \,T_2\right]$, $0< T_1<T_2$. 
\item The reduction is in agreement with the ``classical'' quasi-steady state reduction of system \eqref{sepsys2} for $y^*$:  Assuming QSS for $y^*$ (hence the rate of change for $y^*$ vanishes) one finds
\[
y^*=-G_0(x,0)^{-1}g_1(x,0)+\varepsilon\cdots,
\]
and substitution into the first equation of \eqref{sepsysscal2} yields \eqref{redtnf} up to corrections of order $\varepsilon$, which are irrelevant for Tikhonov reduction. 
\end{enumerate}
\end{proposition}
By Proposition \ref{standardapp}, degenerate scalings yield singular perturbation reductions in the {\em standard setting} of local Tikhonov consistency, with invertible $G_0$ and appropriate eigenvalue conditions. But scaling is not necessary here, since one may directly invoke Tikhonov-Fenichel reduction as stated in \cite{gw2}, Thm.~1. We will outline this approach now; for the reader's convenience we collect some relevant facts used below in the Appendix, in particular Proposition \ref{decompred}. With local Tikhonov consistency system \eqref{sepsys} simplifies to
\begin{equation}\label{sepsys2}
\begin{array}{rcccl}
\dot x&=& F_0(x,y)y&+&\varepsilon f_1(x,y)+\varepsilon^2\cdots,\quad x(0)=x_0\\
\dot y&=& G_0(x,y)y&+&\varepsilon g_1(x,y)+\varepsilon^2\cdots,\quad y(0)=y_0.\\
\end{array}
\end{equation}
\begin{proposition}\label{standardapprem} 
Assume that there exists $\nu>0$ such that all the eigenvalues of $G_0(x,0)$, $x\in U$, have real part $\leq -\nu$.  Then:
\begin{enumerate}[(a)]
\item System \eqref{sepsys2} admits a singular perturbation reduction. The asymptotic slow manifold $\widetilde Z$ is defined, in zero order approximation with respect to $\varepsilon$, by $y=0$. A reduced equation in slow time is given by
\begin{equation}\label{reddirect}
 \begin{array}{rcl}
x^\prime&=&f_1(x,0)-F_0(x,0)G_0(x,0)^{-1}g_1(x,0)\\
y^\prime &=& 0.
\end{array}
\end{equation}
\item In some neighborhood of $\widetilde Z$ there exist $r$ smooth independent first integrals $\psi_1,\ldots,\psi_r$ of the fast system
\[
\begin{array}{rcl}
\dot x&=& F_0(x,y)y\\
\dot y&=& G_0(x,y)y.\\
\end{array}
\]
Whenever $(x_0,y_0)$ lies in the domain of attraction of $\widetilde Z$ then the initial value of the reduced system is locally determined by intersecting $\widetilde Z$ with all the level sets of $\psi_1,\ldots,\psi_r$ that contain $(x_0,y_0)$.
\item Every entry $\psi$ of
\[
x-F_0(x,y)G_0(x,y)^{-1}y
\]
is a first integral of the fast system up to order two in $y$, i.e., the terms in $y$ of orders zero and one of the Lie derivative of $\psi$ vanish. 
In particular, for $y_0=\varepsilon y_0^*$ the initial value of the reduced system equals $x_0$, up to corrections of order $\varepsilon$.
\end{enumerate}
\end{proposition}
\begin{proof}
 The first part is a straightforward application of Proposition \ref{decompred}. One starts from the decomposition $\widehat h^{(0)}=P\cdot\mu$, with
\[
P=\begin{pmatrix}F_0(x,y)\\G_0(x,y)\end{pmatrix},   \quad\mu=y,\quad D\mu=\begin{pmatrix}0&I_s\end{pmatrix},
\]
and obtains the projection matrix
\[
Q= \begin{pmatrix}I_r&0\\ 0& I_s\end{pmatrix}-\begin{pmatrix}F_0(x,0)\\ G_0(x,0)\end{pmatrix}G_0(x,0)^{-1}\begin{pmatrix}0&I_s\end{pmatrix}=\begin{pmatrix}I_r& -F_0(x,0)G_0(x,0)^{-1}\\ 0 & 0\end{pmatrix}.
\]
Then application of $Q$ to 
\[
\begin{pmatrix}f_1(x,0)\\ g_1(x,0)\end{pmatrix}
\]
yields the reduced system.
As for part (b), the existence of the first integrals is shown in Fenichel \cite{fenichel}, Lemma 5.3 and also in \cite{nw11}, Prop.~2.2. The initial value of the reduced equation is locally determined by the first integrals of the fast system according to \cite{gw2}, Prop. 2. To verify part (c), take the Lie derivative of $x-F_0(x,y)G_0(x,y)^{-1}y$ with respect to the fast system to obtain
\[
\left(F_0(x,y)-F_0(x,y)G_0(x,y)^{-1}G_0(x,y)\right)y-\left(\cdots\right)y,
\]
noting that the term in brackets has at least order one in $y$.
\end{proof}
\begin{remark}\label{initialrem}
\begin{enumerate}[(a)]
\item  Proposition \ref{standardapprem} does not require initial value consistency; in this respect its scope is wider. {For instance, by Proposition \ref{standardapprem} one obtains a correct reduction of the linear system in Example \ref{linex}: The fast part of 
\[
\begin{array}{rcl}
\dot x &=&\varepsilon  ax + by\\
\dot y &=& cy
\end{array}
\]
obviously admits the first integral $\psi=cx-by$; and the level set of the first integral containing the initial value $(x_0,y_0)$ intersects the critical manifold in the point $(x_0-by_0/c, 0)$; this is the appropriate initial value for the reduced equation.
}
\item {Proposition \ref{standardapprem} also shows that the scaling $y=\varepsilon y^*$ remains consistent in the long term, since there exists $T_2>0$ such that in the slow time scale $y$ remains bounded on $[0,\,T_2]$, due to Tikhonov's theorem. (More far-reaching properties, such as boundedness on $[0,\,\infty)$, cannot be expected to hold in general.)}
\item The reduced equations in Propositions \ref{standardapp} and \ref{standardapprem} are identical. Moreover, given initial value consistency, the discrepancy of the initial values for the reduced systems is of order $\varepsilon$ and thus does not affect the convergence statement either.
\item Tikhonov's Theorem (as well as \cite{gw2}, Thm. 1) is stated for initial values that are independent of $\varepsilon$, and thus it is not applicable verbatim for a system \eqref{sepsys2} with initial value consistency. But it can obviously be adapted for this case, the point being that it holds for all initial values sufficiently close to the asymptotic slow manifold. 
\item Proposition \ref{standardapprem} provides only the order zero approximation $\widetilde Z$ of the slow manifold. The first order approximation $Z$ may be determined by Proposition \ref{smprop} in the Appendix, and turns out to be
  \begin{align*}
   y=-\varepsilon\, G_0(x,0)^{-1}g_1(x,0)
  \end{align*}
in agreement with Proposition \ref{standardapp}.
\end{enumerate}
\end{remark}

{To summarize: Given the hypotheses of Proposition \ref{standardapp} one may invoke degenerate scalings, but this just amounts to a singular perturbation reduction with critical manifold $y=0$ for the unscaled system. There seems to be no particular benefit from the scaling approach, and moreover one is burdened with the additional requirement of initial value consistency. However, matters turn out to be different in nonstandard settings.}

\subsection{Reduction: The nonstandard case}\label{secnonstandard}
Here we will consider system \eqref{sepsys}, with local Tikhonov consistency and initial value consistency for the scaling but we do not assume invertibility of $G_0(x,0)$ (in the notation of \eqref{sepsys2}). We will discuss under what conditions the transformed system admits a singular perturbation reduction, and determine the reduction. The ``classical'' QSS reduction, as in Proposition \ref{standardapp}(c), is not applicable here. \\
We first illustrate by example that degenerate scaling may be useful when $G_0(x)$ is not invertible. 
\begin{example}\label{mmivex}
\em{Consider the three-dimensional Michaelis-Menten equation \eqref{mmirr}. Assuming initial value consistency for $e$ and $c$, scaling yields the system
\begin{equation}\label{mmirrsc}
\begin{array}{rcl}
\dot{s}&=&\varepsilon(-k_1e^*s+k_{-1}c^*)\\
\dot{e^*}&=&-k_1e^*s +(k_{-1}+k_2)c^*\\
\dot{c^*}&= &k_1e^*s -(k_{-1}+k_2)c^*\\
\end{array}
\end{equation}
in Tikhonov standard form, with
\[
G_0=\begin{pmatrix}-k_1s & k_{-1}+k_2\\
                                  k_1s & -(k_{-1}+k_2)\end{pmatrix}
\]
not invertible. {(The unscaled system does not admit a reduction by Proposition \ref{decompred}, since the direct sum condition (ii) is not satisfied.)} However, there exists a Tikhonov-Fenichel reduction which we now determine. Adopting the notation from the Appendix, Proposition \ref{decompred}, we have
\[
h^{(0)}=\begin{pmatrix}0\\ 1\\ -1\end{pmatrix}\cdot\left(k_1e^*s-(k_{-1}+k_2)c^*\right)=:P\cdot \mu, \quad h^{(1)}=\begin{pmatrix}-k_1e^*s+k_{-1}c^*\\ 0 \\0\end{pmatrix}.
\]
With $D\mu\cdot P=-(k_1s+k_{-1}+k_2)<0$  we obtain a reduction to a two-dimensional system on the manifold defined by $\mu=0$. (Tikhonov yields no immediate reduction to a one-dimensional equation.) Straightforward computations yield the projection matrix
\[
Q=\frac1{d}\cdot\begin{pmatrix}d&0&0\\ -k_1e^* & k_{-1}+k_2 &k_{-1}+k_2\\
                                                                   k_1e^* & k_1s & k_1s\end{pmatrix} \qquad \text{with  }d:=k_1s+k_{-1}+k_2, 
\]
and the reduced system
\begin{equation}\label{mmcomplicated}
\begin{array}{rcl}
s^\prime&=& -k_1e^*s+k_{-1}c^*\\
{e^*}^\prime&=&-\frac{k_1e^*}{d}\left(-k_1e^*s+k_{-1}c^*\right)\\
{c^*}^\prime&=&\phantom{-}\frac{k_1e^*}{d}\left(-k_1e^*s+k_{-1}c^*\right).\\
\end{array}
\end{equation}
In view of $\mu=0$ we may substitute $c^*=k_1e^*s/(k_{-1}+k_2)$. Moreover system \eqref{mmcomplicated} admits the first integral $e^*+c^*$, and the additional relation $e^*+c^*=e_0^*$ finally provides the familiar Michaelis-Menten equation for $s$. \\
In contrast, consider a variant of \eqref{mmirr}, viz.
\begin{equation}\label{mmirrdeg}
\begin{array}{rcl}
\dot{s}&=&-k_1es+k_{-1}c\\
\dot{e}&=&-k_1es +(k_{-1}+k_2)c-\varepsilon \delta^*e\\
\dot{c}&= &k_1es -(k_{-1}+k_2)c\\
\end{array}
\end{equation}
with slow degradation of free enzyme. The matrix $G_0$ is as before, and the same scaling as before now yields the reduced system
\begin{equation}\label{mmcomplicateddeg}
\begin{array}{rcl}
s^\prime&=& -k_1e^*s+k_{-1}c^*\\
{e^*}^\prime&=&-\frac{k_1e^*}{d}\left(-k_1e^*s+k_{-1}c^*+\frac{\delta^*(k_{-1}+k_2)}{k_1}\right)\\
{c^*}^\prime&=&\phantom{-}\frac{k_1e^*}{d}\left(-k_1e^*s+k_{-1}c^*-\delta^*s\right)\\
\end{array}
\end{equation}
on the manifold defined by $\mu=0$. Substituting for $c^*$ one obtains the two-dimensional system
\[
\begin{array}{rcl}
s^\prime&=&-\frac{k_1k_2}{k_{-1}+k_2}e^*s\\
{e^*}^\prime&=&-\frac{k_1}{k_1s+k_{-1}+k_2}e^*\cdot\left(-\frac{k_1k_2}{k_{-1}+k_2}e^*s+\delta^*\frac{k_{-1}+k_2}{k_1}\right)\\
\end{array}
\]
for which no further explicit reduction is apparent.\\
There is a crucial difference between these systems: For system \eqref{mmirr} one may first  employ conservation of $e+c$, and then reduce as in Proposition \ref{standardapp} or \ref{standardapprem}.  But this shortcut is not available for system \eqref{mmirrdeg}.}
\end{example}
One may consider the above as instances of a higher order reduction procedure which starts with a degenerate scaling and provides a reduction with the dimension of the asymptotic slow manifold greater than $r=n-s$ (notation as in \eqref{parteq}.) Generally one can state this as follows.

\begin{proposition}\label{nonstprop}
Consider system \eqref{sepsys} with $f_0(x,0)=0$ and $g_0(x,0)=0$, thus
\begin{equation}\label{intermed}
\begin{array}{rcl}
\dot x&=& F_0(x,y)y+\varepsilon f_1(x,y)+\varepsilon^2f_2(x,y)+\cdots\\
\dot y&=& G_0(x,y)y+\varepsilon g_1(x,y)+\varepsilon^2g_2(x,y)+\cdots
\end{array}
\end{equation}
and also assume initial value consistency.
\begin{enumerate}[(a)]
\item The scaled system then is given by
\begin{equation}\label{intermedsc}
\begin{array}{lcccc}
\dot x&=&   && \varepsilon (f_1(x,0)+F_0(x,0)y^*)+\cdots\\
\dot y^*&=&G_0(x,0)y^*+g_1(x,0)&+&\varepsilon \bar g_1(x,y^*)+\cdots,\\
\end{array}
\end{equation}
with
\[
   \bar g_1(x,y^*):=g_2(x,0)+D_2g_1(x,0)y^*+2D_2^2G_0(x,0)(y^*,y^*),
\]
and $D_2^2G_0$ denoting the second derivative of $G_0$ with regard to $y$.
\item {When $0<s_1:={\rm rank}\,G_0(x,\,0)<s$ for all $x$ in an open subset $\widetilde U\subseteq U$ then near every $x_0\in U$ there exist $\widetilde G_0(x)\in \mathbb R^{(s,s_1)}$ and $R(x)\in\mathbb R^{(s_1,s)}$, both of rank $s_1$, such that
\[
\begin{pmatrix}0\\ G_0(x,0)\end{pmatrix}=\begin{pmatrix} 0\\ \widetilde G_0(x)\end{pmatrix}\cdot R(x).
\]
If there exists a function $w(x)$ on $\widetilde U$ such that $g_1(x,0)=G_0(x)\,w(x)$ then the
 equation $R(x)(y^*+w(x))=0$  determines an $(n-s_1)$--dimensional submanifold $\widehat Z$}, and \eqref{intermedsc} locally admits a Tikhonov-Fenichel reduction with an attracting asymptotic slow manifold $\widehat Z$ if and only if there is a $\nu>0$ such that all eigenvalues of $R\cdot \widetilde G_0$ have real parts $\leq -\nu$.
\end{enumerate}
\end{proposition}
\begin{proof} Part (a) is straightforward from Taylor expansion with respect to $y$.
It is sufficient to prove part (b) for a neighborhood of any $\widehat x\in U$. Here we may assume that the last $s_1$ columns of $G_0$ are linearly independent, and define $\widetilde G_0\in \mathbb R^{(s,s_1)}$ as the matrix with these columns. Then the matrix consisting of the first $s-s_1$ columns of $G_0$ (each being a linear combination of the columns of $\widetilde G_0$) may be expressed in the form $\widetilde G_0\cdot A$, with $A(x)\in\mathbb R^{(s_1,s-s_1)}$ for every $x$, and the first assertion follows from
\[
G_0=\widetilde G_0\cdot \left( A\,,\, I_{s_1}\right)=:\widetilde G_0\cdot R.
\]
{Concerning the existence criterion for a Tikhonov-Fenichel reduction see Proposition \ref{decompred} and \cite{gw2}, Remark 4, noting 
\[
P(x,y^*)=\begin{pmatrix}0\\ \widetilde G_0(x)\end{pmatrix},\quad \mu(x,y^*)=R(x)(y^*+w(x)), 
\]
and 
\[
D\mu(x,y^*)=\left(\ast,\, R(x)\right), \quad D\mu(x,y^*)P(x,y^*)=R(x)\cdot \widetilde G_0(x).
\]
}
\end{proof}
{We give some illustrative examples.}
\begin{example}\label{mmivexcont}\em{
\begin{itemize}
\item  Continuing Example \ref{mmivex} and following the steps in the proof above, one obtains
\[
G_0=\begin{pmatrix}-k_1s & k_{-1}+k_2\\
                                  k_1s & -(k_{-1}+k_2)\end{pmatrix} =\begin{pmatrix}k_{-1}+k_2\\
                                   -(k_{-1}+k_2)\end{pmatrix}\cdot\begin{pmatrix}\frac{-k_1s}{k_{-1}+k_2}, & 1\end{pmatrix};
\]
a slightly different version of the decomposition.
\item To illuminate the effect of scaling on critical manifolds, consider a system
\[
\begin{pmatrix}0&A_{12}\\ 0 & A_{22}\end{pmatrix}\cdot\begin{pmatrix}x\\y\end{pmatrix}+O(\varepsilon)
\]
with constant matrices.
Upon scaling this becomes
\[
\begin{pmatrix}0&0\\ 0 & A_{22}\end{pmatrix}\cdot\begin{pmatrix}x\\y^*\end{pmatrix}+\begin{pmatrix}0\\g_1(x,0)\end{pmatrix}+O(\varepsilon)
\]
For sake of simplicity, let $g_1(\ast,0)=0$. Then, whenever the kernel of $A_{22}$ does not have full rank and its kernel is not contained in the kernel of $A_{12}$, the dimension of the critical manifold increases.
\end{itemize}
}
\end{example}
We refer to section \ref{appsec} for application-relevant examples.
\begin{remark}\label{justscal}
Proposition \ref{nonstprop} also guarantees consistency of the scaling  $y^*=y/\varepsilon$ in the long term. Indeed, by Tikhonov there exist a neighborhood of the critical manifold and some $T_2>0$ such that any solution of  \eqref{intermedsc} with initial value in this neighborhood converges uniformly to a solution of the reduced system on any closed subinterval of $(0,T_2]$. In particular this implies boundedness of $y^*$ for $0\leq t\leq T_2$.
Moreover, if the hypotheses of Hoppensteadt's Theorem (see Hoppensteadt \cite{Hoppensteadt}, as well as a specialization in  \cite{lws} for autonomous systems) are satisfied und thus uniform convergence holds on every closed subinterval of $(0,\infty)$, the same argument shows that the scaling is consistent for all positive times $\tau$.
\end{remark}

\subsection{Conservation of first integrals}
First integrals of singularly perturbed systems are generally preserved (if possibly trivialized) by Tikho\-nov-Fenichel reductions; see Appendix, Proposition \ref{ersteint}. In the following, we examine the effect of degenerate scaling on first integrals.

\begin{proposition}\label{firstint}
Let system \eqref{sepsys} be given, with local consistency and initial value consistency. Furthermore assume that the scaled system \eqref{intermedsc} admits a Tikhonov-Fenichel reduction in the sense of Proposition \ref{decompred} in the Appendix, with reduced system
\begin{equation}\label{rederstint}
\begin{array}{rcccl}
 x'&=& p(x,y^*)\\
 {y^*}'&=& q(x,y^*) \\
\end{array}
\end{equation}
(in slow time) on the asymptotic slow manifold $\widetilde Z$. 
\begin{enumerate}[(a)]
\item Whenever $\varphi$ is smooth on a neighborhood of $U\times V \times [0,\varepsilon_0]$ and $\varphi(\cdot,\,\cdot,\,\varepsilon)$ is a first integral of \eqref{sepsys} for every $0<\varepsilon <\varepsilon_0$, the function $\widetilde\varphi\colon\widetilde Z\to \R,\ \widetilde\varphi(x,y^*):= \varphi(x,0,0)$ is a first integral (possibly constant) of \eqref{rederstint}. 
\item Moreover, if 
  \begin{equation}\label{hilfe}
   \varphi(x,\varepsilon y^*,\varepsilon)=\varepsilon \varphi^*(x,y^*)+\varepsilon^2(\cdots)
  \end{equation}
 holds for all $x,\,y^*,\,\varepsilon$ then $\varphi^*(x,y^*)$ is a first integral of \eqref{rederstint} (possibly constant) on $\widetilde Z$.
\end{enumerate}
\end{proposition}

\begin{proof}
 By Proposition \ref{ersteint}, a first integral $\psi(x,y^*,\varepsilon)$ of \eqref{intermedsc} induces a first integral $\widetilde \psi(x,y^*):=\psi(x,y^*,0)$ of \eqref{rederstint}. Since $\widehat \varphi(x,y^*,\varepsilon):=\varphi(x,\varepsilon y^*,\varepsilon)$ is a first integral of \eqref{intermedsc}, the first assertion follows.\\ If condition \eqref{hilfe} is satisfied then $\varepsilon^{-1}\varphi(x,y^*,\varepsilon)$ is smooth and constant or a first integral of \eqref{intermedsc}. Hence $\varphi^*(x,y^*)$ is a first integral of \eqref{rederstint}.
\end{proof}

\begin{remark}
\begin{enumerate}[(a)]
\item The requirement that $\varphi$ is smooth can be relaxed considerably; see Appendix, Remark \ref{firstintrem}. But in applications to reaction networks one usually deals with smooth first integrals.
\item A conserved first integral that fits part (b) of the Proposition appears in Example \ref{mmivex}.
\end{enumerate}
\end{remark}


\section{Determination of consistent scalings}
Up to this point we assumed that some degenerate scaling of a system was given in advance. Now we turn to determining all such scalings for a general parameter dependent polynomial differential equation
\begin{equation}\label{genpar}
\dot x= h(z,\pi),\quad z\in\mathbb R^n,\, \pi\in\mathbb R^m.
\end{equation}
We first recall that all singular perturbation reductions of such systems can be determined. In order to rewrite such a system the form \eqref{specpar}, fix a ``suitable'' parameter value $\widehat \pi$, choose some $\rho\in\mathbb R^m$ and set
\[
\widehat h(z,\varepsilon):=h(z,\widehat \pi +\varepsilon\rho) =h(z,\widehat\pi)+\varepsilon \cdots.
\]
{In the present context, a parameter value $\widehat\pi$ is ``suitable'' if the perturbed system admits a reduction via the classical theorem of Tikhonov and Fenichel. Following \cite{gwz}, where precise definitions and characterizations are given, we speak of {\em Tikhonov-Fenichel parameter values (TFPV)}.  By \cite{gwz}, all TFPV can be determined, in principle, using methods from algorithmic algebra. But this general method becomes unfeasible for large $n$ or $m$, hence more restricted approaches and heuristics remain relevant. The search for locally Tikhonov consistent scalings a priori amounts to searching for critical manifolds that are coordinate subspaces. For chemical reaction networks this particular approach is motivated by QSS for chemical species. }\\
When a system of the form \eqref{specpar} and a ``small parameter'' are given at the start, there remains the relatively straightforward task of checking all coordinate subspaces for applicability of Proposition \ref{standardapp} or Proposition \ref{nonstprop}. 
For the general case \eqref{genpar} one has to combine possible scalings with an a priori designation of Tikhonov-Fenichel parameter values; we consider just one special instance of the latter.

\subsection{Given small parameter}
 {We consider a system
\eqref{specpar}. This system may originate from \eqref{genpar} by fixing a TFPV, or from a system with known reaction rates by choosing some threshold that separates large and small parameters. We include the case that $\widehat h=\widehat h^{(0)}$ is independent of $\varepsilon$. Due to Definition \ref{consdef}, finding all locally Tikhonov consistent (LTC) scalings is equivalent to finding all partitions}
\[ 
z=\begin{pmatrix}x\\ y\end{pmatrix} \text{  such that   }\,\widehat h^{(0)}(\begin{pmatrix}x\\ 0\end{pmatrix})=0,
\]
possibly involving a relabelling of entries. We write
\[
\widehat h^{(0)}(z)=\left(\sum_{d\geq 1}\,\,\sum_{\ell_1+\cdots+\ell_n=d} {a^i}_{d,\ell_1,\ldots,\ell_n}z_1^{\ell_1}\cdots z_n^{\ell_n}\right)_{1\leq i\leq n},
\]
with real coefficients ${a^i}_{d,\ell_1,\ldots,\ell_n}$, starting from $d= 1$ since constant terms preclude the existence of LTC scalings. Given $0<s<n$ and $1\leq j_1<\cdots <j_s\leq n$, we call $\{j_1,\ldots,j_s\}$ an {\em LTC index set} (and we call $\{z_{j_1},\ldots, z_{j_s} \}$ an {\em LTC variable set}) whenever $z_{j_1}=\cdots z_{j_s}=0$ implies that $\widehat h^{(0)}(z)=0$. We briefly speak of either as an {\em LTC set}. In the partitioning \eqref{parteq} of $z$ the LTC variables correspond to $y$, hence one may think of them as fast variables. A straightforward search for LTC sets would require the discussion of $2^n-2$ cases; an improvement is based on some simple observations.
\begin{remark} \label{ltcrem}
\begin{enumerate}[(a)] 
\item When the monomial $z_1^{k_1}\cdots z_n^{k_n}$ appears in $\widehat h^{(0)}$ with a nonzero coefficient then every LTC index set contains some $j$ such that $k_j>0$. 
\item In particular, any variable that appears in the linear part of $\widehat h^{(0)}$ with a nonzero coefficient is necessarily an LTC variable.
\item Conversely, the complement $\{ i_1,\ldots, i_r\}$ of any LTC index set is characterized by the property that every nonconstant monomial in $z_{i_1},\ldots,z_{i_r}$ appears with coefficient zero. 
\item To find complements of LTC sets, first determine the set $S$ of all indices $i$ such that all powers $z_i^k$ occur only with coefficient $0$ in $\widehat h^{(0)}$. Then the complement of $S$ is contained in every LTC set; in other words, slow variables necessarily lie in $S$.
\item Every superset of an LTC set is (equal to $\{1,\ldots, n\}$ or) an LTC set. Therefore one may focus on minimal LTC sets (or maximal complements).
\end{enumerate}
 \end{remark}
\begin{example} Reactions of first and second order. {\em Many chemical reaction networks involve only reactions of order one or two. For the corresponding systems of degree two one may proceed as follows.
\begin{itemize}
\item First determine the set $S$ according to Remark \ref{ltcrem} (d).
\item For every $i\in S$ the complement of $\{i\}$ is an LTC set.
\item Extending complements: Let $\{i_1,\ldots,i_p\}\subseteq S$ be complementary to an LTC set, thus all $z_{i_k}z_{i_\ell}$, $k,\ell\in\{1,\ldots p\}$ appear only with coefficient $0$. If there is $i_{p+1}\in S$ such that  all $z_{i_k}z_{i_{p+1}}$, $k\in\{1,\ldots p\}$ appear only with coefficient $0$ then the complement of $\{i_1,\ldots,i_p, i_{p+1}\}$ is an LTC set. Otherwise $\{i_1,\ldots,i_p\}$ is maximal, and its complement is a minimal LTC set.

\end{itemize}
}
\end{example}
We give a small example for illustration.
\begin{example}{\em  Given the three-dimensional Michaelis-Menten system \eqref{mmirr}, designate the complete right-hand side as the fast part. Then $S=\{e,s\}$, and since $es$ occurs with nonzero coefficient, $\{e\}$ and $\{s\}$ are the only maximal complements to LTC sets. Thus we find LTC sets $\{e,c\}$ (see Example \ref{mmivex}), and $\{s,c\}$ , which also leads to a Tikhonov-Fenichel reduction (with trivial reduced equation).}
\end{example}
The observations in Remark \ref{ltcrem} obviously lay the groudwork for an algorithmic approach, which will be taken up elsewhere. We also note that initial value consistency may further restrict the possible LTC variables. 
\subsection{Unknown small parameter}
Generally we have a parameter dependent polynomial system \eqref{genpar}, 
written as
\[
h(z,\pi)=\left(\sum_{d\geq 0}\,\,\sum_{\ell_1+\cdots+\ell_n=d} {p^{(i)}}_{d,\ell_1,\ldots,\ell_n}(\pi)z_1^{\ell_1}\cdots z_n^{\ell_n}\right)_{1\leq i\leq n},
\]
with polynomials ${p^{(i)}}_{d,\ell_1,\ldots,\ell_n}$. From this vantage point there exists no a priori small/large separation of coefficients, and the objective is to simultaneously determine candidates $\widehat\pi$ for Tikhonov-Fenichel parameter values and corresponding LTC variable sets. The summation starts at $d=0$ here, since constant terms may appear in $h$ (although they must vanish at $\widehat\pi$). We do not attempt a complete analysis here but just make a few observations.
\begin{remark}
\begin{enumerate}[(a)]
\item Let $\widehat\pi$ be a TFPV of \eqref{genpar} with LTC variable set $\{z_{j_1},\ldots, z_{j_s} \}$. Then ${p^{(i)}}_{d,k_1,\ldots,k_n}(\widehat\pi)=0$ whenever all $q$ with $k_q>0$ lie in the complement of $\{{j_1},\ldots, {j_s} \}$. Hence the corresponding coefficient of $z_1^{k_1}\cdots z_n^{k_n}$ is of order $\varepsilon$ in $h(z,\widehat\pi+\varepsilon\rho)$. From this observtion one may start an exhaustive case-by-case analysis (with ascending degree $d=0,\,1,\,2\ldots$).
\item From an applied perspective it is of interest to consider a restricted approach with
{\em pre-assigned LTC variables:} Given fixed indices $1\leq j_1<\cdots<j_s\leq n$, determine $\widehat\pi$ so that $\{j_1,\ldots,j_s\}$ is an LTC index set; equivalently ${p^{(i)}}_{d,k_1,\ldots,k_n}(\widehat\pi)=0$ whenever $k_{j_1}=\cdots =k_{j_s}=0$. In different terminology, determine parameter conditions such that $z_{j_1},\ldots,z_{j_s}$ are in steady state at $\widehat \pi$, hence in quasi-steady state at $\pi=\widehat\pi+\varepsilon\pi^*+\cdots$.
\end{enumerate}
\end{remark}
The algebraic problem of finding the common roots $\widehat\pi$ of a given collection of ${p^{(i)}}_{d,k_1,\ldots,k_n}$ is amenable to algorithmic methods, for instance using Groebner bases; compare \cite{gwz3}, Appendix A 5.2 (in particular Remark 4). \\
Again we consider the Michaelis-Menten system to illustrate the approach.
\begin{example}{\em 
For system \eqref{mmirr} with $\pi=(k_1,k_{-1}, k_2)$, we discuss all possible assignments of LTC variable sets.
\begin{enumerate}[(i)]
\item  Pre-assigned LTC variable set $\{ s\}$, thus $s=\varepsilon s^*$. Then the first equation of  \eqref{mmirr} forces $\widehat k_{-1}=0$; no further conditions need to be imposed. The scaled system
\[
\begin{array}{rcl}
\dot{s^*}&=&-k_1es^*+k^*_{-1}c\\
\dot{e}&=&-\varepsilon k_1es^* +(\varepsilon k^*_{-1}+k_2)c\\
\dot{c}&= &\varepsilon k_1es^* -(\varepsilon k^*_{-1}+k_2)c\\
\end{array}
\]
admits a singular perturbation reduction (via Proposition \ref{decompred}) to the critical manifold defined by $s^*=c=0$, with trivial reduced system. There exists no Tikhonov-Fenichel reduction with critical manifold given by $e=c=0$.
\item  Pre-assigned LTC variable set $\{ e\}$, thus $e=\varepsilon e^*$. The second equation of  \eqref{mmirr} forces $\widehat k_{-1}+\widehat k_2=0$, hence $\widehat k_{-1}=\widehat k_2=0$ due to nonnegativity. This yields $k_{-1}=\varepsilon k^*_{-1}$ and $k_2=\varepsilon k^*_2$. The system admits a singular perturbation reduction with critical manifold given by $k_1e^*s=(k^*_{-1}+k^*_2)c$.
\item  Pre-assigned LTC variable set $\{ c\}$, thus $c=\varepsilon c^*$. The third equation in \eqref{mmirr} forces $\widehat k_1=0$; the scaled system admits a Tikhonov-Fenichel reduction.
\item   Pre-assigned LTC variable set $\{ s, e\}$, thus $s=\varepsilon s^*$ and $e=\varepsilon e^*$. One obtains the further conditions $k_{-1}=\varepsilon k_{-1}^*$ and $k_2=\varepsilon k_2^*$, but no Tikhonov-Fenichel reduction exists.
\item   Pre-assigned LTC variable set $\{ s, c\}$, thus $s=\varepsilon s^*$ and $c=\varepsilon c^*$. No further conditions need to be imposed on the rate constants; there exists a singular perturbation reduction with trivial reduced equation.
\item   Pre-assigned LTC variable set $\{ e, c\}$, thus $e=\varepsilon e^*$ and $c=\varepsilon c^*$. No further conditions need to be imposed on the rate constants; the reduction is known from Example \ref{mmivex}.
\end{enumerate}
}
\end{example}
To summarize, the approach produces hits and misses, but the point is that it is systematic.

\section{Applications}\label{appsec}
\subsection{Intermediate species in PTM networks}
Intermediates in reaction networks are of particular interest for the computation of stationary points, and also for the computation of quasi-steady state reductions, because they allow linear elimination of variables; see Feliu and Wiuf \cite{fewi12,fewi13},  Gunawarenda \cite{gunaw}, and Saez et al. \cite{sawife}. Recall that in classical QSS reduction, a (formally) reduced system is obtained via algebraic elimination of fast species. We discuss these reductions from the perspective of degenerate scalings, including convergence issues.

By definition and in view of Remark \ref{ltcrem}(b), an intermediate species must correspond to a LTC variable whenever it appears with nonzero coeficient in the fast part of the reaction system. \\
We will discuss post-translational modification (PTM) networks in some detail, based on Feliu and Wiuf \cite{fewi13}, Gunawardena \cite{gunaw}. The species in a PTM network are of two types: Substrates (with concentrations we call $u_1,\ldots, u_N$) and intermediates (with concentrations $v_1,\ldots, v_P$). The differential equation for a PTM system with mass action kinetics has the form
\begin{equation}\label{ptmeq}
\begin{array}{rcccl}
\dot u_i&=&\sum_{j=1}^N\sum_{k=1}^P\theta_{ij}\left(-a_{ij}^ku_iu_j+b_{ij}^kv_k\right)&+&\sum_{j=1}^N\left(d_{ji}u_j-d_{ij}u_i\right)\\
\dot v_k&=&\sum_{j=1}^N\sum_{i=1}^j \left(a_{ij}^ku_iu_j-b_{ij}^kv_k\right)&+&\sum_{\ell=1}^P\left(c_{\ell k}v_\ell-c_{k\ell}v_k\right)\\
\end{array}
\end{equation}
with the coefficients satisfying the following conditions:
\begin{enumerate}[(i)]
\item$\theta_{ij}=1$ whenever $i\not=j$, and $\theta_{ii}=2$; 
\item Rate constants: All $a_{ij}^k=a_{ji}^k\geq 0$, all  $b_{ij}^k=b_{ji}^k\geq 0$; moreover all $d_{ij}\geq 0, \,d_{ii}=0$ and $c_{k\ell}\geq 0,\,c_{kk}=0$;
\item Every $v_k$ appears on the right hand side of \eqref{ptmeq} with a nonzero coefficient;
\item Ordering of substrates: There is an $m$,  $1\leq m\leq N$ such that $i\in\{1,\ldots,m\}$ if and only if $d_{ij}=d_{ji}=0$ for all $i$. 
\end{enumerate}
Feliu and Wiuf \cite{fewi13} focus on stationary points of the differential equation system, but their central arguments work equally well for quasi-steady state. Gunawarenda \cite{gunaw} also mentions QSS for such systems, and the general results in Saez et al. \cite{sawife} apply as well. \\
We focus our interest on networks with no slow reactions, thus the fast part is the whole system. Then the following hold.
\begin{itemize}
\item By Remark \ref{ltcrem} (iii) and (iv), every $v_k$ is necessarily an LTC variable, as well as $u_{m+1}, \ldots,u_N$; the set $S$ consists of $u_1,\ldots, u_m$.
\item It follows from the definition of a cut (\cite{fewi13}, Def.~1) that every LTC variable set (resp., every minimal LTC variable set) is the union of $\{v_1,\ldots,v_P\}$ and a cut (resp., a minimal cut) of $\{u_1,\ldots, u_n\}$. The results in \cite{fewi13} regarding elimination of species (in particular Proposition 2 about elimination of intermediates and Proposition 6 about elimination of certain substrates) may therefore be employed in the (a priori formal) computation and study of QSS reductions.
\item The results from Subsections \ref{secstandard} and \ref{secnonstandard} now allow to discuss the convergence of the reduction in rather general circumstances. A priori, PTM networks will always lead to a non-standard settings, since $\sum u_i+2\sum v_k$ is a linear first integral for \eqref{ptmeq}. But when there exist sufficiently many linear first integrals to eliminate species (see the detailed discussion in \cite{fewi13}, Subsections 3.3 and 3.4) then one ends up with a standard case. For standard cases the QSS reduction agrees with the singular perturbation reduction (up to irrelevant higher order terms) by Proposition \ref{standardapp}, and convergence is guaranteed whenever the critical manifold is linearly attractive for the fast system.
\end{itemize}
\begin{example}{\em 
Consider a generalization of the irreversible Michaelis-Menten reaction scheme with three intermediate complexes:
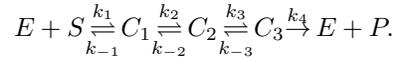
\[
E+S{\overset{k_1}{\underset{k_{-1}}\rightleftharpoons}}
C_1{\overset{k_2}{\underset{k_{-2}}\rightleftharpoons}} C_2 {\overset{k_3}{\underset{k_{-3}}\rightleftharpoons}}C_3 {\overset{k_{4}}\rightarrow}\,E+P.
\]
The associated differential equation is
\[
\begin{array}{rcccccccl}
\dot e&=&-k_1es&+&k_{-1}c_1 & &   &+& k_4c_3\\
\dot s&=& -k_1es&+& k_{-1}c_1 &&&&\\
\dot c_1&=& k_1es&-&(k_{-1}+k_2)c_1&+&k_{-2}c_2 && \\
\dot c_2&=&  & & k_2c_1 &-&(k_{-2}+k_3)c_2&+& k_{-3}c_3 \\
\dot c_3 &=& &&&& k_3c_2&-&(k_{-3}+k_4)c_3
\end{array}
\]
and the initial values are $e(0)=e_0$, $s(0)=s_0$ and all $c_i(0)=0$.
\begin{enumerate}[(a)]
\item When all reactions are fast then necessarily $c_1, c_2$ and $c_3$ are LTC variables. We augment these by $e$ to obtain a minimal LTC variable set, and we then have to impose $e_0=\varepsilon e_0^*$ to ensure initial value consistency. With the first integral $e+c_1+c_2+c_3$ one obtains a system for $s$ and the $c_i$; for this the standard reduction is applicable. The eigenvalue condition can be verified via the Hurwitz-Routh conditions (see Gantmacher \cite{Gant}) for the characteristic polynomial of the matrix (using the notation of Proposition \ref{decompred})
\[
D\mu(z)P(z)=\begin{pmatrix}-(k_1s+k_{-1}+k_2) & -k_1s+k_{-2}& -k_1s\\
                                          k_2&-(k_{-2}+k_3) & k_{-3}\\
                                          0 & k_3 & -(k_{-3}+k_4)\end{pmatrix}.
\]
Computation of the reduced equation is straightforward: Since the critical manifold is the affine subspace defined by $c_1=c_2=c_3=0$, Tikhonov-Fenichel and classical QSS reduction are in agreement by \cite{gwz3}, Proposition 5. Solving for $\dot c_i=0$ from bottom up, and substituting in the equation for $s$, one obtains
\[
\dot s=-\frac{k_1k_2k_3k_4e_0}{d}s
\]
with 
\[
\begin{array}{rcl}
d&=& (k_1k_2k_3+k_1k_2k_4+k_1k_2k_{-3}+k_1k_3k_4+k_1k_4k_{-2}+k_1k_{-2}k_{-3})s \\
    &&+k_3k_4k_{-1}+k_4k_{-1}k_{-2}+k_{-1}k_{-2}k_{-3}.
\end{array}
\]
\item On the other hand consider the network with slow last reaction, thus $k_4=\varepsilon k_4^*$. The fast part of the system is given by
\[
\begin{pmatrix}
-k_1es+k_{-1}c_1 \\
 -k_1es+ k_{-1}c_1\\
 k_1es-(k_{-1}+k_2)c_1+k_{-2}c_2 \\
k_2c_1 -(k_{-2}+k_3)c_2+ k_{-3}c_3 \\
k_3c_2-k_{-3}c_3
\end{pmatrix}.
\]
Again, we have the first integral $e+c_1+c_2+c_3=e_0$ and the $c_i$ are necessarily LTC variables. Augmenting by $e$ (with $e_0=\varepsilon e_0^*$) yields a minimal LTC set, but now (by a straightforward computation) the singular perturbation reduction leads to $\dot s=0$. This disagrees with the classical QSS reduction of the system for $s, \,c_1,\,c_2,\,c_3$ with the $c_i$ in quasi-steady state, which actually is incorrect. (Proposition 5 from \cite{gwz3} is not applicable here, since the critical manifold is not an affine subspace.)  In fact, the four-dimensional system with $k_4=\varepsilon k_4^*$ (but $e_0$ not small) does admit a singular perturbation reduction which can be computed via Proposition \ref{decompred}, with a critical manifold that is not affine. This example illustrates the limitations of the scaling approach, with its restricted choices for slow and fast variables. 
\end{enumerate}
}
\end{example}
\begin{example}{\em  We discuss the Michaelis-Menten network with inhibitor, in a variant that includes slow degradation of the inhibitor:
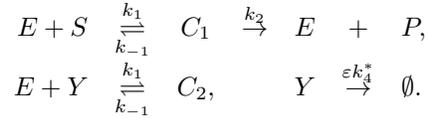
\[
\begin{array}{rcccccl}
E+S&{\overset{k_1}{\underset{k_{-1}}\rightleftharpoons}}&
C_1&{\overset{k_{2}}\rightarrow}&E&+&P,\\
E+Y&{\overset{k_1}{\underset{k_{-1}}\rightleftharpoons}}&
C_2,&  &Y&{\overset{\varepsilon k_{4}^*}\rightarrow}&\emptyset.
\end{array}
\]
The corresponding differential equation is 
\[
\begin{array}{rclcl}
\dot e &=&-k_1es+(k_{-1}+k_2)c_1 & -&k_3ey+k_{-3}c_2\\
\dot s&=& -k_1es+k_{-1}c_1 & &  \\
\dot c_1&=& k_1es-(k_{-1}+k_2)c_1 & &  \\
\dot c_2&=&                                     && k_3ey-k_{-3}c_2\\
\dot y&=&                                     &-&k_3ey+k_{-3}c_2 -\varepsilon k_4^* y\\
\end{array}
\]
with $c_1(0)=c_2(0)=0$, all other initial values positive. Considering the fast part, $c_1$ and $c_2$ are necessarily LTC variables, and $\{e,c_1,c_2\}$ is a minimal LTC set, with initial value consistency forcing $e_0=\varepsilon e_0^*$. The scaled system does not admit standard reduction but the nonstandard approach works. One has
\[
h^{(0)}= \begin{pmatrix} -1&-1\\ 0&0\\ 1&0\\ 0&1\\ 0&0\end{pmatrix}\cdot \begin{pmatrix} k_1e^*s-(k_{-1}+k_2) c_1^*\\ k_3e^*y-k_{-3}c_2^*\end{pmatrix}=:P\cdot \mu;
\]
the critical manifold $Z$ is given by $\mu=0$, equivalently
\[
c_1^*=\frac{k_1e^*s}{k_{-1}+k_2},\quad c_2^*=\frac{k_3e^*y}{k_{-3}};
\]
moreover 
\[
h^{(1)}=\begin{pmatrix}0\\ -k_1e^*s+k_{-1}c_1^*\\ 0\\0\\ -k_3e^*y+k_{-3}c_2^*-k_4^*y\end{pmatrix}=\begin{pmatrix}0\\ -k_2c_1^*\\ 0\\0\\ -k_4^*y\end{pmatrix} \quad\text{on  }Z.
\]
One readily verifies the eigenvalue condition for the $2\times 2$ matrix $D\mu\cdot P$. The second and the fifth rows of $P$ are zero, therefore the same holds for the second and the fifth rows of $P(D\mu P)^{-1}D\mu$, whence the second and fifth entries of $Q\cdot h^{(1)}$ are equal to those of $h^{(1)}$. From the first integral $e^*+c_1^*+c_2^*$ (see Proposition \ref{firstint}) and the defining equations for $Z$ one sees that only the second and fifth entries of the reduced system are needed. This yields the reduced system
\[
\begin{array}{rcl}
s^\prime &=&-k_2c_1^*\\
y^\prime&=& -k_4^*y
\end{array}
\]
on $Z$, with
\[
e^*\left(1+\frac{k_1}{k_{-1}+k_2}s+\frac{k_3}{k_{-3}}y\right)=e_0^*; \quad c_1^*=\frac{k_1}{k_{-1}+k_2}se^*.
\]
Setting
\[
M_1:=\frac{k_{-1}+k_2}{k_1}, \quad M_2:=\frac{k_3}{k_{-3}}
\]
one finally obtains
\[
\begin{array}{rcl}
s^\prime &=&-\frac{k_2e_0^*s}{M_1+s+M_1M_2y}\\
y^\prime&=& -k_4^*y
\end{array}.
\]
The familiar reduction (see e.g. Keener and Sneyd \cite{KeSn}, or \cite{GSWZ1}) when inhbitor does not degrade consists of just the first equation, with $y=y_0$. Hence the result of our reduction is intuitively obvious, but still a rigorous derivation and a convergence proof are preferable.
}
\end{example}
\subsection{Reaction-transport systems}
Ordinary differential equations which model reactions and transport appear in various circumstances, for instance as spatial discretizations of reaction-diffusion systems, or in multicellular reaction networks. We derive some general results for scalings of such systems, and look at some applications. (We note in  passing that the results may be of some use for partial differential equations modelling reaction-transport systems.) \\  We start with a locally Tikhonov consistent ``reaction system''
\[
\begin{array}{rcl}
\dot x&=& F(x,y)y+\varepsilon f(x,y,\varepsilon)\\
\dot y&=& G(x,y)y+\varepsilon g(x,y,\varepsilon),
\end{array}
\]
reactions taking place in each of $N$ compartments (numbered by $\alpha$). We denote the concentrations in compartment $\#\alpha$ by
$x_\alpha, \,y_\alpha$, and furthermore set
\[
x_\alpha=\begin{pmatrix}x_{1,\alpha}\\ \vdots \\ x_{r,\alpha}\end{pmatrix},\quad y_\alpha=\begin{pmatrix}y_{1,\alpha}\\ \vdots \\ y_{s,\alpha}\end{pmatrix},\quad \widehat x=\left( x_\alpha\right)_{1\leq \alpha\leq N}\text{  and  } \widehat y=\left( y_\alpha\right)_{1\leq \alpha\leq N}.
\]
A reaction-transport system with general transport terms is then given by
\begin{equation}\label{rtfull}
\begin{array}{rcl}
\dot x_\alpha&=& F(x_\alpha,y_\alpha)y_\alpha+\varepsilon f(x_\alpha, y_\alpha,\varepsilon)+\Delta^{(x)}(\varepsilon)\sum_{\beta=1}^N  \Theta_{\alpha,\beta}^{(x)}(\widehat x, \widehat y,\varepsilon)x_\beta\\
\dot y_\alpha&=& G(x_\alpha,y_\alpha)y_\alpha+\varepsilon g(x_\alpha,y_\alpha,\varepsilon)+\Delta^{(y)}(\varepsilon)\sum_{\beta=1}^N  \Theta_{\alpha,\beta}^{(y)}(\widehat x, \widehat y,\varepsilon)y_\beta.\\
\end{array}
\end{equation}
Here the $ \Theta_{\alpha,\beta}^{(x)}$ and $ \Theta_{\alpha,\beta}^{(y)}$ are smooth matrix-valued functions of appropriate sizes, and
\[
\Delta^{(x)}(\varepsilon)={\rm diag}\,(\varepsilon^{\rho_1},\ldots,\varepsilon^{\rho_r}),\quad \Delta^{(y)}(\varepsilon)={\rm diag}\,(\varepsilon^{\sigma_1},\ldots,\varepsilon^{\sigma_s})
\]
with all $\rho_i,\,\sigma_j\in\{0,\,1\}$. Thus $\rho_i=0$ means fast transport for all $x_{i,\alpha}$ while $\rho_i=1$ means slow transport, and analogously for the $y_{j,\alpha}$. (Thus the scaling $\widehat y=\varepsilon \widehat y^*$ is not necesarily locally consistent.) The general form of the transport terms only implies that no transport occurs when no species are present in any compartment; in concrete applications one will impose stronger requirements. \\
The existence of transport terms strongly influences possible LTC variables: Whenever some $\Delta^{(x)}(0)\cdot\Theta_{\alpha,\beta}^{(x)}(0,0,0)\not=0$ then some entry of $x_\beta$ is necessarily an LTC variable, and whenever some $\Delta^{(y)}(0)\cdot\Theta_{\alpha,\beta}^{(y)}(0,0,0)\not=0$ then some entry of $y_\beta$ is necessarily an LTC variable.\\
 The following special result is a straightforward consequence of Proposition \ref{nonstprop}.
\begin{proposition}
When $\rho_1=\cdots=\rho_r=1$ (thus transport of every $x_\alpha$ is slow) then system \eqref{rtfull} is locally Tikhonov consistent with respect to the scaling $\widehat y=\varepsilon {\widehat y}^*$. The scaled fast subsystem is 
\begin{equation}\label{rtsfast}
\begin{array}{rcccl}
\dot x_\alpha&=& 0& &\\
\dot y^*_\alpha&=& G(x_\alpha,0)y^*_\alpha+ g(x_\alpha,0,0)&+&\Delta^{(y)}(0)\sum_{\beta=1}^N  \vartheta_{\alpha,\beta}^{(y)}(\widehat x)y^*_\beta\\
\end{array}
\end{equation}
with
\[
 \vartheta_{\alpha,\beta}^{(y)}(\widehat x):=\Theta_{\alpha,\beta}^{(y)}(\widehat x,0,0).
\]
\end{proposition}
It will depend on further properties whether this scaled system admits a Tikhonov-Fenichel reduction. 

\begin{example}\label{examplemmdiff}
{\em  The reduction of a
spatially discretized Michaelis-Menten system with slow diffusion was investigated in \cite{flww}. This is a system of dimension $3N$, with $N$ compartments, given by 
\begin{equation}\label{mmdiff}
\begin{array}{rclcl}
\dot s_\alpha&=& \varepsilon \theta_s^* D_\alpha \widehat s &-& k_1s_\alpha w_\alpha+(k_1s_\alpha+k_{-1})c_\alpha,\\
\dot c_\alpha&=& \varepsilon \theta_c^* D_\alpha\widehat  c &+& k_1s_\alpha w_\alpha-(k_1s_\alpha+k_{-1}+k_2)c_\alpha,\\
\dot w_\alpha&=& \varepsilon \theta_e^* D_\alpha\widehat  w &+& \varepsilon(\theta_c^*-\theta_e^*)D_\alpha \widehat c
\end{array}
\end{equation}
with $1\leq \alpha\leq N$, { $w_\alpha:=e_\alpha+c_\alpha$, and $\widehat x:=(x_\alpha)$ for $x\in\{s,\,c,\,w\}$. The matrix $D$ represents discretized diffusion, and $D_\alpha$ denotes the $\alpha^{\rm th}$ row of $D$. 
System \eqref{mmdiff} is in Tikhonov standard form and one readily determines the reduced equation (in slow time $\tau=\varepsilon t$) as
\[
w_\alpha^\prime=\theta_e^*D_\alpha \widehat w,\quad 1\leq \alpha\leq N
\]
on the $N$-dimensional critical manifold defined by $\widehat c=\widehat s=0$. This simply describes diffusion with all reactions in equilibrium.\\
In order to discuss reduction of the system under the familiar assumption of small initial overall enzyme concentration $w_\alpha(0)=\varepsilon e_{\alpha,0}^*$, one must introduce scalings. Following Proposition \ref{nonstprop} we let $\widehat c=\varepsilon \widehat c^*$ and $\widehat w=\varepsilon \widehat w^*$. In scaled variables we then obtain 
\begin{equation}\label{mmdiffscal}
\begin{array}{rcl}
\dot s_\alpha&=& \varepsilon \left(\theta_s^* D_\alpha\widehat s- k_1s_\alpha w_\alpha^*+(k_1s_\alpha+k_{-1})c_\alpha^*\right)\\
\dot c_\alpha^*&=& \varepsilon  \theta_c^* D_\alpha\widehat c ^*+ k_1s_\alpha w_\alpha^*-(k_1s_\alpha+k_{-1}+k_2)c_\alpha^*\\
\dot w_\alpha^*&=& \varepsilon\left( \theta_e^* D_\alpha \widehat w^* + (\theta_c^*-\theta_e^*)D_\alpha\widehat c^*\right)
\end{array}
\end{equation}
with $1\leq \alpha\leq N$,
and the critical manifold is given by 
\[
 k_1s_\alpha w_\alpha^*-(k_1s_\alpha+k_{-1}+k_2)c_\alpha^*=0, \quad 1\leq \alpha\leq N.
\]
 This system is in standard form, and one obtains the following reduced equation in slow time:
\begin{equation}\label{mmdiffred}
\begin{array}{rcl}
s_\alpha^\prime&=& \theta_s^*D_\alpha\widehat s -\frac{k_1k_2w^*_\alpha s_\alpha}{k_1s_\alpha+k_{-1}+k_2}\\
{w_\alpha^*}^\prime&=&\theta_e^*D_\alpha\widehat w^* +(\theta_c^*-\theta_e^*)D_\alpha\left((\frac{k_1k_2w^*_\beta s_\beta}{k_1s_\beta+k_{-1}+k_2})_{1\leq \beta\leq N}\right).
\end{array}
\end{equation} 
See \cite{flww} (where a slightly different scaling was employed) for details, and for extending the reduction to the reaction-diffusion PDE.} \\
}

\end{example}
\begin{example}\label{extrans}
{\em 
 Multicellular reaction networks describe (identical) reactions in $N$ cells that are connected by transport (see Shapiro and Horn \cite{shapiro} for linear transport terms, Korc and Feinberg \cite{kf} for polynomial mass action reaction terms). We consider the simple reaction
  \[
   S+P\xrightleftharpoons[k_{-1}]{k_{1}} C.
  \]
 Denote by $(s_{\alpha},p_{\alpha},c_{\alpha})$ the concentrations of $S$, $P$ and $C$ in cell $\alpha$ and set
  \[
   (\hat s,\hat p,\hat c):=(s_{\alpha},p_{\alpha},c_{\alpha})_{1\leq\alpha\leq N}.
  \]
\begin{enumerate}[a)]
 \item Assuming mass action kinetics and slow transport, the multicellular reaction network can be described by the system
  \begin{align*}
   &\dot s_{\alpha}=-k_1s_{\alpha}p_{\alpha}+k_{-1}c_{\alpha}+\varepsilon\sum_\beta\Theta^{(s)}_{\alpha, \beta}(\widehat s,\widehat p,\widehat c,\varepsilon)s_\beta \\
   &\dot p_{\alpha}=-k_1s_{\alpha}p_{\alpha}+k_{-1}c_{\alpha}+\varepsilon\sum_\beta\Theta^{(p)}_{\alpha, \beta}(\widehat s,\widehat p,\widehat c,\varepsilon)p_\beta\\
   &\dot c_{\alpha}=\ \ k_1s_{\alpha}p_{\alpha}-k_{-1}c_{\alpha}+\varepsilon\sum_\beta\Theta^{(c)}_{\alpha, \beta}(\widehat s,\widehat p,\widehat c,\varepsilon)c_\beta.
  \end{align*}
Here each $c_\alpha$ is necessarily a LTC variable; we augment these by the $s_\alpha$, imposing initial value consistency for $\widehat c$ and $\widehat s$ as additional conditions. 
We scale $s_{\alpha}=\varepsilon s^*_{\alpha}$ and $c_{\alpha}=\varepsilon c^*_{\alpha}$ for $1\leq \alpha\leq N$. Therefore
  \begin{align*}
   &\dot s_{\alpha}^*=-k_1s_{\alpha}^*p_{\alpha}+k_{-1}c_{\alpha}^*+\varepsilon\sum_\beta\Theta^{(s)}_{\alpha, \beta}(0,\widehat p,0,0)s^*_\beta +\varepsilon^2\cdots\\
   &\dot p_{\alpha}=-\varepsilon k_1s_{\alpha}^*p_{\alpha}+\varepsilon k_{-1}c_{\alpha}^*+\varepsilon\sum_\beta\Theta^{(p)}_{\alpha, \beta}(0,\widehat p,0,0)p_\beta +\varepsilon^2\cdots\\
   &\dot c_{\alpha}^*=k_1s_{\alpha}^*p_{\alpha}-k_{-1}c_{\alpha}^*+\varepsilon\sum_\beta\Theta^{(c)}_{\alpha, \beta}(0,\widehat p,0,0)c^*_\beta +\varepsilon^2\cdots
  \end{align*}
for $1\leq \alpha<N$. This is a nonstandard case, with Proposition \ref{decompred} applicable to the scaled system. For the reduced system on the critical manifold defined by $k_{-1}c_{\alpha}^*=k_1s_{\alpha}^*p_{\alpha}$ for $1\leq \alpha\leq N$ and with
\[
{Q_\alpha=\frac1{k_1p_\alpha+k_{-1}}\begin{pmatrix}k_{-1} & -k_1s^*_\alpha & k_{-1}\\
                                                                                                         0 & k_1p_\alpha+ k_{-1} & 0 \\
                                                                                                        k_1p_\alpha & k_1s^*_\alpha & k_1 p_\alpha\end{pmatrix}}
\]
 one obtains (in slow time)
\[
\begin{pmatrix}{s^*_\alpha}^\prime \\
                          p_\alpha^\prime \\
                        {c^*_\alpha}^\prime \end{pmatrix}= Q_{{\alpha}}\cdot\begin{pmatrix}\sum_\beta\Theta^{(s)}_{\alpha, \beta}(0,\widehat p,0,0)s^*_\beta\\
                                                                                                                      \sum_\beta\Theta^{(p)}_{\alpha, \beta}(0,\widehat p,0,0)p_\beta\\
                                                                                                                    \sum_\beta\Theta^{(c)}_{\alpha, \beta}(0,\widehat p,0,0)c^*_\beta
                                                                                                   \end{pmatrix}.
\]
Here one may replace $c_{\alpha}^*=k_1s_{\alpha}^*p_{\alpha}/k_{-1}$ and discard the equations for $c^*_\alpha$; see also \cite{gl}, subsection 5.2.
\item If one considers fast transport for $S$ (which forces the $s_\alpha$ to be LTC variables) but still requires slow transport for the other species, the scaled system becomes
 \begin{align*}
   &\dot s_{\alpha}^*=-k_1s_{\alpha}^*p_{\alpha}+k_{-1}c_{\alpha}^*+\sum_\beta\Theta^{(s)}_{\alpha, \beta}(0,\widehat p,0,0)s^*_\beta +\varepsilon\cdots\\
   &\dot p_{\alpha}=-\varepsilon k_1s_{\alpha}^*p_{\alpha}+\varepsilon k_{-1}c_{\alpha}^*+\varepsilon\sum_\beta\Theta^{(p)}_{\alpha, \beta}(0,\widehat p,0,0)p_\beta +\varepsilon^2\cdots\\
   &\dot c_{\alpha}^*=k_1s_{\alpha}^*p_{\alpha}-k_{-1}c_{\alpha}^*+\varepsilon\sum_\beta\Theta^{(c)}_{\alpha, \beta}(0,\widehat p,0,0)c^*_\beta +\varepsilon^2\cdots.
  \end{align*}
For further discussions we focus on a simple transport mechanism, viz. 
\[
\sum_\beta \Theta^{(s)}_{\alpha, \beta}(\widehat s,\widehat p,\widehat c,\varepsilon)s_\beta =\delta_s\left(s_{\alpha-1}-2s_\alpha+s_{\alpha+1}\right)
\]
with the understanding that $s_0=s_1$ and $s_{N+1}=s_N$; similarly for $\widehat p$ and $\widehat c$. (These terms appear when discretizing diffusion in dimension one, with Neumann boundary conditions). The computations for the reduction are slightly involved; details are given in Appendix \ref{appextrans} below.
With
  \[
   \sum_\beta\Theta^{(s)}_{\alpha, \beta}(0,\widehat p,0,0)s^*_\beta +\varepsilon\cdots=\delta_s\left(s^*_{\alpha-1}-2s^*_\alpha+s^*_{\alpha+1}\right),
  \]
the reduced equation turns out as
  \begin{align*}
   &{s_{\alpha}^*}'=0\\
   &{p_{\alpha}}'=\delta_p\left(p_{\alpha-1}-2p_\alpha+p_{\alpha+1}\right)\\
   &{c_{\alpha}^*}'=\frac{k_1s_{\alpha}}{k_{-1}}\delta_p\left(p_{\alpha-1}-2p_\alpha+p_{\alpha+1}\right)
  \end{align*}
on the critical manifold defined by $s^*_1=\ldots=s^*_N$ and $k_1s^*_{\alpha}p_{\alpha}=k_{-1}c^*_{\alpha}$. This can be simplified further to
  \[
   {p_{\alpha}}'=\delta_p\left(p_{\alpha-1}-2p_\alpha+p_{\alpha+1}\right)
  \]
(pure diffusion of $\widehat p$) augmented by 
  \[
   s^*_{\alpha}=\widetilde s^*_{0}:=\frac{k_{-1}\sum_{\alpha=1}^N(s^*_{\alpha,0}+c^*_{\alpha,0})}{k_{-1}N+k_1\sum_{\alpha=1}^Np_{\alpha,0}} \quad \text{and} \quad c^*_{\alpha}=\frac{k_1\widetilde s^*_{0}}{k_{-1}}p_{\alpha}
  \]
 where $\left(s^*_{\alpha,0},p_{\alpha,0},c^*_{\alpha,0}\right)_{1\leq \alpha\leq N}$ are the initial values of the scaled system. Thus, fast transport of substrate alone is sufficient to reduce the system to a diffusion problem.
\end{enumerate}
}
\end{example}

\appendix

\section{Appendix: Tikhonov-Fenichel reductions}
For the reader's convenience we restate here some variants of \cite{gw2}, Thm.~1 and Remark 2 on singular perturbation reductions with no a priori separation of slow and fast variables. In addition we determine the first-order approximation of the slow manifold and prove a conservation property for first integrals in the reduction procedure. {The necessary theoretical background from singular perturbation theory does not go beyond the classical work of Tikhonov \cite{tikh} and Fenichel \cite{feninv,fenichel}, but the approach from \cite{gw2} is quite convenient for explicit computations.}\\

Given an open subset $S\subset \R^n$ and a smooth function $h$ defined on a neighborhood of $ S\times [0,\varepsilon_0]$ we consider the system 
 \begin{equation}\label{evolution}
   \dot z= h(z,\varepsilon)= h^{(0)}(z)+\varepsilon h^{(1)}(z)+\varepsilon^2 \cdots
  \end{equation}
in the asymptotic limit $\varepsilon\to 0$, as well as its time-scaled version 
\begin{equation}\label{evolutionscaled}
 z^\prime=  \frac{d z}{d\tau}=\varepsilon^{-1}h(z,\varepsilon)=\varepsilon^{-1}h^{(0)}(z)+  h^{(1)}(z)+\ldots, \quad \tau=\varepsilon t.
  \end{equation}
We impose the following requirements, which are necessary and sufficient for the local existence of a transformation into Tikhonov standard form (see \cite{nw11}):
\begin{enumerate}[(i)]
\item  There exists a point $z_0$ in the zero set $\mathcal V(h^{(0)})$ such that ${\rm rank} \,D h^{(0)}(z)=r<n$ for all $z$ in some neighborhood of $z_0$ in $\mathbb R^n$. By the implicit function theorem, there exists a neighborhood $U$ of $z_0$ such that $Z:=U\cap \mathcal V(h^{(0)})$ is a $(n-r)$-dimensional submanifold.
\item  There is  a direct sum decomposition
  \[
   \mathbb R^n={\rm ker}\ D h^{(0)}(z) \oplus {\rm im}\ D h^{(0)}(z)
  \]
for all $z\in Z$.
\item There is $\nu>0$ such that all nonzero eigenvalues of $D h^{(0)}(z)$, $z\in U$,  have real part $\leq -\nu$.
\end{enumerate}
A coordinate-free local version of Tikhonov's and Fenichel's reduction theorem can be stated as follows.
\begin{proposition}\label{decompred} Let conditions (i)--(iii) be given. 
\begin{enumerate}[(a)]
\item{\em  Decomposition.} On some neighborhood $\widetilde U\subseteq U$ of $z_0$ there exist smooth maps  \[P\colon \widetilde U\to \mathbb R^{n\times r}\quad \text{and} \quad \mu\colon\widetilde U\to \mathbb R^r\]  with ${\rm rank}\, P(z_0)={\rm rank} \,D\mu(z_0)=r$, such that
    \[
      h^{(0)}(z)=P(z)\mu(z),\quad z\in \widetilde U.
    \]
 Moreover, the zero set $Y$ of $\mu$ satisfies $Y=Z\cap \widetilde U= \mathcal V( h^{(0)})\cap\widetilde U$. The entries of $\mu$ may be taken as any $r$ entries of $h^{(0)}$ that are functionally independent at $z_0$.
\item {\em Reduction.} The system
    \begin{equation}\label{Grenzsystem}
     z^\prime=q(z):=Q(z)\cdot h^{(1)}(z)
    \end{equation}
  with the projection matrix \[Q(z):=Id-P(z)(D\mu(z)P(z))^{-1}D\mu(z)\] 
is defined in $\widetilde U$. Moreover, every entry of $\mu$ is a first integral of \eqref{Grenzsystem}; in particular it admits the manifold $Y$ as an invariant set.
\item {\em Convergence.} There exists $T_2>0$ and a neighborhood $U^*\subset U$ of $Y$ such that all solutions of \eqref{evolutionscaled} starting in $U^*$ converge to solutions of the reduced system \eqref{Grenzsystem} on $Y$ as $\varepsilon\to 0$, uniformly on $[T_1,T_2]$ for any $T_1$ with $0<T_1<T_2$.
\end{enumerate}
\end{proposition}
We call \eqref{Grenzsystem} the Tikhonov-Fenichel reduction of \eqref{evolutionscaled}.  
As follows from \cite{gw2}, for rational $h^{(0)}$ one may choose $P$ and $\mu$ rational, and the decomposition can be obtained via algorithmic algebra.\\
The submanifold $Z$ is called the {\em asymptotic slow manifold} (or {\em critical manifold}) of the system. 
For small $\varepsilon$ there exists an invariant slow manifold of system \eqref{evolution} for which $Z$ is an order $\varepsilon^0$ approximation. For the sake of completeness we also include an order $\varepsilon^1$ approximation:
\begin{proposition}\label{smprop}
Up to first order in $\varepsilon$, the slow manifold of system \eqref{evolution} is determined by any equation
\[
 \mu(z)+\varepsilon\left((D\mu(z)P(z))^{-1}\cdot D\mu(z) \,h^{(1)}(z)+ A(z)\cdot\mu(z)\right)=0
\]
with an arbitrary smooth matrix-valued function $A$, up to higher order terms.
\end{proposition}

\begin{proof} 
A zero order approximation is given by
$\mu(z)=0$.
For a first order approximation we make the ansatz
\[
\Phi(z)=\mu(z)+\varepsilon\Psi(z)=0.
\]
According to (for instance) \cite{gwz3}, Lemma 2, the invariance condition with $h(z)=h^{(0)}(z)+\varepsilon h^{(1)}(z)+\cdots$ is then
\[
D\Phi(z)\cdot h(z)=\left(\Lambda_0(z)+\varepsilon\Lambda_1(z)\right)\cdot \Phi(z) +\varepsilon^2\cdots
\]
with suitable smooth matrix-valued functions $\Lambda_i$.
Recall that 
\[
h^{(0)}(z)=P(z)\cdot \mu(z) 
\]
and that $D\mu(z)P(z)$ is invertible. Evaluation of the invariance condition yields in order zero:
\[
D\mu(z) P(z)\cdot\mu(z)=\Lambda_0(z)\cdot \mu(z);\text{  hence we may set  }\Lambda_0(z)=D\mu(z)P(z).
\]
In order one we obtain
\[
D\mu(z)\,h^{(1)}(z)+D\Psi(z)\,P(z)\cdot \mu(z)=D\mu(z)P(z)\cdot\Psi(z) +\Lambda_1(z)\cdot\mu(z).
\]
This equation has the particular solution
\[
\Psi_0(z)=(D\mu(z)P(z))^{-1}\cdot D\mu(z) \,h^{(1)}(z),\quad \Lambda_1(z)=D\Psi_0(z)\,P(z),
\]
and general solution
\begin{align*}
 &\Psi(z)=\Psi_0(z)+A(z)\mu(z),\\ &\Lambda_1(z)=D\Psi_0(z)\,P(z)+D\left(A(z)\mu(z)\right)P(z)-D\mu(z)P(z)A(z)
\end{align*}
with an arbitrary smooth matrix-valued function $A$, which is irrelevant for the approximation. Indeed, $I_N +\varepsilon A$ is invertible for small $\varepsilon$, hence 
\[
\mu(z)+\varepsilon\left(\Psi_0(z) +A(z)\cdot \mu(z)\right)=\left(I_n+\varepsilon A(z)\right)\cdot\left(\mu(z)+\varepsilon\Psi_0(z) +\varepsilon^2)\cdots\right),
\]
and the determining equations (viz., $\mu(z)+\varepsilon \Psi_0(z)=0$ and $\mu(z)+\varepsilon \Psi(z)=0$) agree up to order $\varepsilon$. 
\end{proof}

Finally, we prove that first integrals of \eqref{evolution} (equivalently, of \eqref{evolutionscaled}) are preserved by the reduction. (Special instances of this appear in the literature; see e.g. Feliu and Wiuf \cite{fewi12}, but there seems to be no record of a general result.) The proof is straightforward via Proposition \ref{decompred}.

\begin{proposition}\label{ersteint}
Let the assumptions of Proposition \ref{decompred} be satisfied, and let $\varphi=\varphi(z,\varepsilon)$ be smooth on a neighborhood of $S \times [0,\varepsilon_0]$ such that $\varphi(\cdot,\varepsilon)$ is a first integral of \eqref{evolutionscaled} for $0<\varepsilon <\varepsilon_0$. Then $\widetilde\varphi\colon S\to \R,\ \widetilde \varphi(z):=\varphi(z,0)$ is constant on $Z$ or a first integral of the reduced system \eqref{Grenzsystem} on $Z$ (i.e., the intersection of $Z$ with any level set of $\widetilde\varphi$ is invariant for the reduced system).
\end{proposition}

\begin{proof}
 Let $z_0$ be in the domain of attraction of $Z$, denote by $z(\tau,\varepsilon)$ the solution of \eqref{evolutionscaled} with initial value $z(0)=z_0$ and let $\bar z(\tau)$ denote the solution of \eqref{Grenzsystem} such that $z$ converges to $\bar z$ on every closed subset $[T_1,T_2]$ of $(0,T_2]$ as $\varepsilon\to0$. Hence, on $[T_1,T_2]$ we have 
  \[
   \widetilde\varphi(\bar z(\tau))=\varphi(\bar z(\tau),0)\xleftarrow[\varepsilon\to0]{} \varphi(z(\tau,\varepsilon),\varepsilon)=\varphi(z_0,\varepsilon)\xrightarrow[\varepsilon\to0]{} \varphi(z_0,0)
  \]
 by continuity. This implies that $\widetilde \varphi$ is constant on $[0,T_2]$. Since (in particular) every point of $Z$ may be taken as initial value for \eqref{evolutionscaled}, we see that $\widetilde\varphi$ is constant on any solution of \eqref{Grenzsystem} on $Z$.
\end{proof}
\begin{remark}\label{firstintrem}
\begin{enumerate}[(a)]
\item Obviously Proposition \ref{ersteint} holds with less restrictive assumptions on $\varphi$. For instance, requiring continuity in $(z,\varepsilon)$ and continuous differentiability in $z$ for all $\varepsilon$ suffices.
\item The hypotheses of Proposition \ref{ersteint} have rather strong implications for a system \eqref{tikhostandard} in Tikhonov standard form: Given a first integral
\[
\varphi(x,y)=\varphi_0(x,y)+\varepsilon\varphi_1(x,y)+\cdots,
\]
with $\varphi_0$ not constant (w.l.o.g.), we obtain the condition
\[
D_2\varphi_0(x,y)\widetilde g_0(x,y)=0
\]
by comparing lowest order terms in $\varepsilon$. Fix $x=x^*$, let $(x^*,y^*)\in Z$ and consider the equation
\[
\dot y=\widetilde g_0(x^*,y).
\]
The eigenvalue requirement on $D_2\widetilde g_0$ implies that the stationary point $y^*$ of this equation is linearly asymptotically stable, hence the first integral $\varphi(x^*,y)$ must be constant in some neighborhood of $y^*$. We see that $\varphi_0$ is independent of $y$. Morever, comparing higher order terms in $\varepsilon$ will in general show that no (nonconstant) first integrals exist. These remarks indicate for the case of standard reduction that no nonconstant first integrals exist, but matters are different for nonstandard reduction.
\end{enumerate}
\end{remark}

\section{Appendix: Some computations}\label{appextrans}
We return to Example \ref{extrans} b), hence
 \begin{align*}
   &\dot s_{\alpha}^*=-k_1s_{\alpha}^*p_{\alpha}+k_{-1}c_{\alpha}^*+\delta_s\left(s^*_{\alpha-1}-2s^*_\alpha+s^*_{\alpha+1}\right)\\
   &\dot p_{\alpha}=-\varepsilon k_1s_{\alpha}^*p_{\alpha}+\varepsilon k_{-1}c_{\alpha}^*+\varepsilon\delta_p\left(p_{\alpha-1}-2p_\alpha+p_{\alpha+1}\right)\\
   &\dot c_{\alpha}^*=k_1s_{\alpha}^*p_{\alpha}-k_{-1}c_{\alpha}^*+\varepsilon\delta_c\left(c^*_{\alpha-1}-2c^*_\alpha+c^*_{\alpha+1}\right)
  \end{align*}
(with $s_0=s_1$ and $s_{N+1}=s_N$; similarly for $\widehat p$ and $\widehat c^*$), and compute its reduction. The fast part is 
  \[
   h^{(0)}=\begin{pmatrix}
            \mu_{R}+\delta_s\mathcal D \widehat s^*\\ 0 \\ -\mu_R
           \end{pmatrix}
  \]
with
  \[
   \mu_R:=\left(-k_1s^*_{\alpha}p_{\alpha}+k_{-1}c^*_{\alpha}\right)_{1\leq \alpha\leq N}
  \]
and
  \[
   \mathcal D \widehat s^*:=\left(s^*_{\alpha-1}-2s^*_\alpha+s^*_{\alpha+1}\right)_{1\leq \alpha\leq N}.
  \]
We fix the ordering $s^*_1,\ldots,s^*_N,p_1,\ldots,p_N,c^*_1,\ldots,c^*_N$ of the variables. The critical manifold is defined by $s^*_1=\ldots=s^*_N$ and $k_1s^*_{\alpha}p_{\alpha}=k_{-1}c^*_{\alpha}$ for all $\alpha$, 
hence we let
  \[
 \mu=\begin{pmatrix}\mu_R\\ \mu_D\end{pmatrix} \quad\text{with}\quad \mu_D:=\begin{pmatrix}
          s^*_2-s^*_1 \\ \vdots \\ s^*_N-s^*_{N-1}
         \end{pmatrix}.
  \]
The slow part (restricted to the critical manifold) is given by
  \[
   h^{(1)}=\begin{pmatrix}
            0 \\ \delta_p\mathcal D\widehat p\\ \delta_c\mathcal D\widehat c^*
           \end{pmatrix}.
  \]
In order to obtain a convenient form for the decomposition of $h^{(0)}$ we
define 
  \begin{align*}
   &D_1:=-k_1 \diag(p_1,\ldots,p_N),\\ &D_2:=-k_1 \diag(s^*_1,\ldots,s^*_N),\\ &D_3:=k_{-1} I_N,
  \end{align*}
and 
 \[
  M_1:=\begin{pmatrix}
      -1 & 1 & &&\\
         & -1 & 1&&\\
         & & \ddots & \ddots  &\\
         &&& -1 &1
     \end{pmatrix}\in\R^{(N-1)\times N},  \quad M_2:=-M_1^{\rm tr}\in\R^{N\times (N-1)},
 \]
furthermore
  \[
   P:=\begin{pmatrix}
       I_N & M_2 \\ 0 & 0 \\ -I_N & 0
      \end{pmatrix}.
  \]
Then one has a decomposition $h^{(0)}=P\cdot\mu$ in the sense of Proposition \ref{decompred}, and
  \[
   D\mu=\begin{pmatrix}
         D_1 & D_2 & D_3 \\ M_1 & 0 & 0
        \end{pmatrix}\Longrightarrow     D\mu P=\begin{pmatrix}
         D_1-D_3 & D_1M_2 \\ M_1 & M_1M_2
        \end{pmatrix}.
  \]
It is not necessary to invert this matrix in order to find the projection: According to Goeke \cite{godiss}, Bem.~2.1.9 it suffices to find some $\gamma=(\gamma_1,\gamma_2)^{\tr}\in \R^{N+(N-1)}$ such that
  \begin{equation}\label{hilfeapp}
   D\mu\cdot P\gamma=D\mu\cdot h^{(1)}.
  \end{equation}
Then the reduced system is then given as
\begin{align*}
 &{\begin{pmatrix}
   \widehat s^* \\ \widehat p \\ \widehat c^*
  \end{pmatrix}}'
=h^{(1)}-P\cdot \gamma.
\end{align*}
Equation \eqref{hilfeapp} is equivalent to the system
  \begin{equation}\label{auxsys}
\begin{array}{rcl}
   (D_1-D_3)\gamma_1+D_1M_2\gamma_2&=&\Delta\\
   M_1\left(\gamma_1+M_2\gamma_2\right)&=&0
\end{array}
  \end{equation}
with 
  \[
   \Delta:=D_2\delta_p\mathcal D\widehat p+D_3\delta_c \mathcal D\widehat c^*.
  \]
The second equation in \eqref{auxsys} is solved by $\gamma_1+M_2\gamma_2=0$, and substitution into the first equation yields $D_3\gamma_1=-\Delta$. Thus there exists a solution with
  \begin{align*}
   &\gamma_1=-(k_{-1})^{-1}\Delta\\
   &P_D\gamma_2=(k_{-1})^{-1}\Delta.\\
  \end{align*}
The reduced system on the critical manifold is given by
\begin{align*}
 &{\begin{pmatrix}
   \widehat s^* \\ \widehat p \\ \widehat c^*
  \end{pmatrix}}'
=\begin{pmatrix}
  -\gamma_1-M_2\gamma_2 \\ \delta_p \mathcal D(\widehat p)\\ \delta_c \mathcal D(\widehat c^*)+\gamma_1
 \end{pmatrix}
=\begin{pmatrix}
  0 \\ \delta_p \mathcal D(\widehat p) \\ k_1k_{-1}^{-1}s_0^*\delta_p \mathcal D(\widehat p)
 \end{pmatrix}
\end{align*}
with $s_0^*:=s^*_1=\ldots=s^*_N$. The condition $k_1s_{\alpha}p_{\alpha}=k_{-1}c^*_{\alpha}$ together with constancy of $\widehat s^*$  now shows that this is a discrete diffusion equation
  \[
   {p_{\alpha}}'=\delta_p\left(p_{\alpha-1}-2p_\alpha+p_{\alpha+1}\right)
  \]
augmented by $s^*_{\alpha}=\widetilde s^*_{0}$ and $c^*_{\alpha}=\frac{k_1\widetilde s^*_{0}}{k_{-1}}p_{\alpha}$, with $\widetilde s^*_{0}$ the corresponding initial value of $s^*_{\alpha}$ on the critical manifold. This initial value is determined according to \cite{gw2}, Prop.~2: We show that
  \[
   \widetilde s_0^*=\frac{\sum_{\alpha=1}^N(s^*_{\alpha,0}+c^*_{\alpha,0})}{N+k_1k_{-1}^{-1}\sum_{\alpha=1}^Np_{\alpha,0}},
  \]
where $\left(s^*_{\alpha,0},p_{\alpha,0},c^*_{\alpha,0}\right)_{1\leq \alpha\leq N}$ are the initial values of the scaled system. Indeed
  \begin{align*}
   &\psi_{1,\alpha}=p_{\alpha},\quad 1\leq \alpha\leq N\\
   &\psi_{2}=\sum_{\alpha=1}^N(s^*_{\alpha}+c^*_{\alpha})
  \end{align*}
are first integrals of the fast system $\dot x =h^{(0)}(x)$. Now Prop.~2 in \cite{gw2} shows that the initial values $\left(\widetilde s^*_{\alpha,0},\widetilde p_{\alpha,0},\widetilde c^*_{\alpha,0}\right)_{1\leq \alpha\leq N}$ on the critical manifold are determined by the conditions
  \begin{align*}
   &\widetilde p_{\alpha,0}=p_{\alpha,0}\\
   &\sum_{\alpha=1}^N(\widetilde s^*_{\alpha,0}+\widetilde c^*_{\alpha,0})=\sum_{\alpha=1}^N(s^*_{\alpha,0}+c^*_{\alpha,0})\\
   &\widetilde s^*_0:=\widetilde s^*_{1,0}=\ldots=\widetilde s^*_{N,0}\\ 
   &k_1\widetilde s^*_{\alpha,0}\widetilde p_{\alpha,0}=k_{-1}\widetilde c^*_{\alpha,0}.
  \end{align*}
The assertion follows.

\end{document}